\documentclass[leqno]{article}
\usepackage{graphicx}
\usepackage{mathtools}
\mathtoolsset{showonlyrefs}
\usepackage{comment}
\usepackage[abbrev,lite,nobysame]{amsrefs}
\usepackage[svgnames]{xcolor}
\usepackage{amsthm,amsfonts,amssymb}
\usepackage{mathrsfs}
\usepackage{hyperref}
\usepackage{cleveref}
\usepackage{bm}
\usepackage{fontawesome}
\usepackage{csquotes}
\usepackage{mathabx}
\usepackage{ulem}
\usepackage{tikz}
\usetikzlibrary{math}
\usepackage{pgfplots}
\usepgflibrary{patterns}
\usetikzlibrary{arrows,shapes,patterns.meta,decorations.pathmorphing,decorations.text,positioning,arrows,decorations.markings,shapes.arrows} 
\usetikzlibrary{arrows.meta, positioning, calc}
\pgfplotsset{compat=1.10}

\newcommand{\eps}{\varepsilon}
\renewcommand{\phi}{\varphi}
\DeclareMathOperator\supp{supp}
\DeclareMathOperator\Span{span}
\newcommand{\R}{\mathbb{R}}
\newcommand{\C}{\mathbb{C}}
\renewcommand{\L}{\mathcal{L}}
\renewcommand{\H}{\mathcal{H}}
\newcommand{\Z}{\mathbb{Z}}

\theoremstyle{plain}
\newtheorem{conj}{Conjecture}[section]
\newtheorem{thm}[conj]{Theorem}
\newtheorem{prop}[conj]{Proposition}
\newtheorem{lemma}[conj]{Lemma}
\theoremstyle{definition}
\newtheorem{dfn}[conj]{Definition}
\newtheorem{remark}[conj]{Remark}

\title{Uniform null-controllability times for the Coron--Guerrero problems are $2$ and $2 + 2 \sqrt{2}$}

\begin{document}\linespread{1.06}\selectfont

	\author{
    Kai Koike\,\footnote{
    Department of Mathematics, Hiyoshi Campus,
    Keio University,
    Yokohama 223-8521, Japan, e-mail: kaikoike.math@keio.jp}
    \and
    Vincent Laheurte\,\footnote{
    Department of Mathematics, Maison du nombre, 6 avenue de la Fonte, University of Luxembourg, L-4364 Esch-sur-Alzette, Luxembourg, E-mail address: vincent.laheurte@uni.lu}
    }
	
	\maketitle

    \tableofcontents

    \newpage

    \begin{abstract}
        We study the uniform null-controllability, in the vanishing viscosity limit, of the transport--diffusion equation $u_t + M u_x = \eps u_{xx}$ on $[0,L]$ with a Dirichlet boundary control at $x=0$. Our results concern the uniform null-controllability time $T_{\mathrm{unif}}$, defined as the infimum of the times for which the null-controllability cost remains bounded as $\eps \to 0$. We determine this time exactly for both signs of the transport velocity: $T_{\mathrm{unif}} = 2 L / M$ in the positive-speed case $M>0$, and $T_{\mathrm{unif}} = (2 + 2 \sqrt{2}) L / |M|$ in the negative-speed case $M<0$. The positive-speed result disproves the conjecture $T_{\mathrm{unif}}=L/M$ suggested by Coron and Guerrero. For negative speed, the conjectured threshold had already been disproved by Lissy. Our result determines the exact threshold. We establish the lower bounds by constructing adjoint solutions that violate uniform observability below the respective thresholds. In the positive-speed case, the construction exploits the asymptotic structure of the Blaschke products and model spaces associated with the exponential family generated by the adjoint spectrum. The same model-space structure is used to reduce the upper-bound problems for both signs of $M$ to infinite-time observability inequalities, which are proved through a representation of the boundary-to-interior map and estimates of its Hilbert--Schmidt norm.
    \end{abstract}

\section{Introduction}
The purpose of this article is to study the uniform null-controllability, in the vanishing viscosity limit, of the transport--diffusion equation
\begin{equation}\label{eq:transport-diffusion}
    \begin{dcases}
        u_t + M u_x = \eps u_{xx} & \text{in $(0,T) \times (0,L)$}, \\
        u(\cdot,0)=f, \quad u(\cdot,L)=0 & \text{in $(0,T)$}, \\
        u(0,\cdot)=u_0 & \text{in $(0,L)$}.
    \end{dcases}
\end{equation}
Here, $u \colon [0,T] \times [0,L] \to \R$ is a scalar unknown, $u_0 \colon [0,L] \to \R$ is the initial data, $M$ is a real number, $L,\eps>0$ are positive constants, and $f \colon [0,T] \to \R$ is the boundary control. By the classical theory of Fattorini and Russell on controllability of 1D linear parabolic equations~\cite{FattoriniRussell1971}, for any initial data $u_0 \in L^2(0,L)$ and time $T>0$, there exists a control $f \in L^2(0,T)$ such that the solution $u$ to~\eqref{eq:transport-diffusion} satisfies $u(T,\cdot) \equiv 0$. Denote by $f_{u_0}$ such a control of minimal $L^2(0,T)$-norm and set
\begin{equation}\label{eq:cost}
    C_{\eps}(T,L,M)
    \coloneqq
    \sup_{\| u_0 \|_{L^2(0,L)} \leq 1}
    \| f_{u_0} \|_{L^2(0,T)}.
\end{equation}
Our main object is the uniform null-controllability time
\begin{equation}\label{eq:def_T_unif}
    T_{\mathrm{unif}}(L,M)
    \coloneqq
    \inf\left\{
        T>0:
        \limsup_{\eps\to0} C_{\eps}(T,L,M)<+\infty
    \right\}.
\end{equation}
Thus $T_{\mathrm{unif}}(L,M)$ is the threshold time beyond which null-controllability can be achieved with a cost that remains bounded in the vanishing viscosity limit. For $M\neq0$, an elementary scaling shows that
\begin{equation}\label{eq:scaling_T_unif}
    T_{\mathrm{unif}}(L,M)
    =
    \frac{L}{|M|}
    T_{\mathrm{unif}}\bigl(1,\operatorname{sgn}M\bigr),
\end{equation}
so that only the sign of $M$ matters after normalization.

There is a fundamental distinction between the cases $M=0$ and $M \neq 0$. When $M=0$, equation~\eqref{eq:transport-diffusion} reduces to the heat equation, for which it is classical that
\begin{equation}
    T_{\mathrm{unif}}(L,0) = +\infty;
\end{equation}
see~Seidman~\cite{Seidman1984} and G\"{u}ichal~\cite{Guichal1985}. In contrast, Coron and Guerrero~\cite{CoronGuerrero2005} proved that $T_{\mathrm{unif}}(L,M)$ is finite if $M \neq 0$. More precisely, in the positive speed case $M>0$, they proved that $TM/L>4.3$ implies $\lim_{\eps \to 0}C_{\eps}(T,L,M)=0$, whereas $\lim_{\eps \to 0}C_{\eps}(T,L,M)=+\infty$ if $TM/L < 1$. Hence
\begin{equation}\label{eq:CG_bounds_Tunif}
    1
    \leq
    \frac{T_{\mathrm{unif}}(L,M) M}{L}
    \leq
    4.3
    \qquad (M>0).
\end{equation}
For $M<0$, they obtained the corresponding bounds
\begin{equation}\label{eq:CG_bounds_Tunif_negative}
    2
    \leq
    \frac{T_{\mathrm{unif}}(L,M) |M|}{L}
    \leq
    57.2
    \qquad (M<0).
\end{equation}
The lower bound in~\eqref{eq:CG_bounds_Tunif} is natural in the sense that the vanishing viscosity limit system $u_t + M u_x = 0$ with a control at $x=0$ is controllable if and only if $TM/L \geq 1$. The situation is more subtle for $M<0$ since the boundary control at $x=0$ does not enter the limiting system as $x=0$ is the outflow boundary. In~\cite[Remark~6]{CoronGuerrero2005}, Coron and Guerrero left open whether the lower bounds in~\eqref{eq:CG_bounds_Tunif} and~\eqref{eq:CG_bounds_Tunif_negative} are optimal, which have subsequently been referred to as the Coron--Guerrero conjectures.

\begin{conj}\label{conj:M>0}
    Let $M>0$. Then
    \begin{equation}
        TM/L>1 \Rightarrow \limsup_{\eps \to 0}C_{\eps}(T,L,M)=0.
    \end{equation}
    In particular, we have
    \begin{equation}
        \frac{T_{\mathrm{unif}}(L,M) M}{L}=1.
    \end{equation}
\end{conj}

\begin{conj}\label{conj:M<0}
    Let $M<0$. Then
    \begin{equation}
        T|M|/L>2 \Rightarrow \limsup_{\eps \to 0}C_{\eps}(T,L,M)=0.
    \end{equation}
    In particular, we have
    \begin{equation}
        \frac{T_{\mathrm{unif}}(L,M) |M|}{L}=2.
    \end{equation}
\end{conj}

The upper bounds in~\eqref{eq:CG_bounds_Tunif} and~\eqref{eq:CG_bounds_Tunif_negative} were progressively improved. Glass~\cite{Glass2010} used complex-analytic tools to replace the constants $4.3$ and $57.2$ by $4.2$ and $6.1$, respectively. Lissy~\cite{Lissy2012} discovered a useful link between the uniform null-controllability problem and the cost of fast controls for the heat equation. Combining this connection with the estimate of Tenenbaum and Tucsnak~\cite{TenenbaumTucsnak2007}, he obtained the constants $2\sqrt{3}$ and $2\sqrt{3}+2$. Dard\'{e} and Ervedoza~\cite{DardeErvedoza2019} subsequently improved the estimate in~\cite{TenenbaumTucsnak2007}, leading to
\begin{equation}\label{eq:previous_upper_bounds}
    \frac{T_{\mathrm{unif}}(L,M) M}{L}
    \leq
    4\sqrt{K_0}
    \simeq
    3.3385
    \qquad (M>0)
\end{equation}
and
\begin{equation}\label{eq:previous_upper_bounds_negative}
    \frac{T_{\mathrm{unif}}(L,M) |M|}{L}
    \leq
    4\sqrt{K_0}+2
    \simeq
    5.3385
    \qquad (M<0),
\end{equation}
where $K_0\simeq0.6966$. Recently, Lissy~\cite{Lissy2026FastHeat} proved the sharpness of the bound in~\cite{DardeErvedoza2019} and identified
\begin{equation}
    K_0=\frac{\Gamma(1/4)^4}{8\pi^3}.
\end{equation}
Thus the upper bounds~\eqref{eq:previous_upper_bounds} and~\eqref{eq:previous_upper_bounds_negative} are the best ones obtainable from the transfer argument of~\cite{Lissy2012}. On the lower-bound side, Lissy~\cite{Lissy2015} showed that, for $M<0$,
\begin{equation}\label{eq:Lissy_negative}
    \frac{T_{\mathrm{unif}}(L,M) |M|}{L}
    \geq
    2\sqrt{2},
\end{equation}
thereby disproving Conjecture~\ref{conj:M<0}.

A related problem in a multidimensional and variable-coefficient setting was studied by Guerrero and Lebeau~\cite{GuerreroLebeau2007}. More recently, Laurent and L\'eautaud~\cite{LaurentLeautaud2021Gradient} showed, for gradient flows, that the uniform null-controllability time can be strictly larger than the minimal controllability time of the limiting transport equation, and related this phenomenon to semiclassical spectral properties. See also~\cite{LaurentLeautaud2023} for a further refinement in the one-dimensional setting.

In this paper, we determine the uniform null-controllability time exactly for both signs of $M$. Our first main result concerns the positive-speed case, and in particular, disproves Conjecture~\ref{conj:M>0}.

\begin{thm}[Positive speed]\label{thm:main_positive}
    Let $M>0$. Then
    \begin{equation}\label{thm:main_positive_blowup}
        TM/L<2 \Rightarrow \lim_{\eps \to 0}C_{\eps}(T,L,M)=+\infty
    \end{equation}
    and
    \begin{equation}\label{thm:main_positive_zero}
        TM/L>2 \Rightarrow \lim_{\eps \to 0}C_{\eps}(T,L,M)=0.
    \end{equation}
    In particular, we have
    \begin{equation}
        \frac{T_{\mathrm{unif}}(L,M) M}{L}=2.
    \end{equation}
\end{thm}

Our results leave open whether the control cost remains bounded as $\eps \to 0$ at the critical time $T = 2 L / M$.

The theorem shows that the minimal null-controllability time $L/M$ for the limiting transport equation is not the correct uniform null-controllability time: the vanishing viscosity problem retains an additional obstruction up to twice that time. It also contrasts with the numerical study of Amirat and M\"unch~\cite{AmiratMunch2019}, which reported evidence consistent with the threshold $T_{\mathrm{unif}}(L,M)M/L=1$ while emphasizing the difficulty of accessing very small $\eps>0$ numerically. In fact, as will become clear from the proof of Theorem~\ref{thm:main_positive}, intricate cancellation of increasingly high modes plays a crucial role in the determination of the threshold $T_{\mathrm{unif}}(L,M) M/L=2$, which makes this discrepancy with finite-resolution computations plausible.

Our second main result determines the uniform null-controllability time in the negative-speed case as well.

\begin{thm}[Negative speed]\label{thm:main_negative}
    Let $M<0$. Then
    \begin{equation}\label{thm:main_negative_blowup}
        T|M|/L<2+2\sqrt{2} \Rightarrow \lim_{\eps \to 0}C_{\eps}(T,L,M)=+\infty
    \end{equation}
    and
    \begin{equation}\label{eq:main_negative}
        T|M|/L \geq 2+2\sqrt{2} \Rightarrow \lim_{\eps \to 0}C_{\eps}(T,L,M)=0.
    \end{equation}
    In particular, we have
    \begin{equation}
        \frac{T_{\mathrm{unif}}(L,M) |M|}{L}=2+2\sqrt{2}.
    \end{equation}
\end{thm}

Thus the negative-speed threshold is $2+2\sqrt{2} \simeq 4.8284$. Moreover, the control cost tends to zero even at the critical time. This improves both the previously known lower bound in~\eqref{eq:Lissy_negative} and the upper bound in~\eqref{eq:previous_upper_bounds_negative}, and closes the gap between them.

The vanishing of the control cost stated in Theorems~\ref{thm:main_positive} and~\ref{thm:main_negative} follows directly from Theorems~\ref{prop:positive_cost_bound} and~\ref{prop:negative_cost_bound}, respectively. These provide explicit exponential decay estimates strictly beyond the critical times. At the critical time in the negative-speed case, Theorem~\ref{prop:negative_cost_bound} gives an $O(\sqrt{\eps})$ bound. The blow-up conclusions follow from Theorems~\ref{thm:counterexample_observation} and~\ref{thm:counterexample_observation_negative}, respectively. The former does not provide an explicit blow-up rate for positive speed, whereas the latter gives an explicit exponentially growing lower bound for negative speed.

We now briefly describe the main ideas of the proof. First, note that by the scaling~\eqref{eq:scaling_T_unif}, we may assume $L=|M|=1$. The first implication~\eqref{thm:main_positive_blowup} in Theorem~\ref{thm:main_positive} is proved through the dual observability problem. As pointed out in~\cite{Glass2010}, the boundary observation can be expressed in terms of the exponential family $(e^{-\lambda_n^{\eps}t})_{n=1}^{\infty}$, where $(\lambda_n^{\eps})_{n=1}^{\infty}$ are the Dirichlet eigenvalues of the adjoint transport--diffusion operator $-(\pm \partial_x + \eps \partial_{x}^{2})$. The key structure used in the analysis is that the associated model spaces in the Hardy space $H^2$ have a nontrivial limit as $\eps \to 0$. More precisely, the relevant Blaschke products converge to the inner function $e^{-2s}$. Under the Laplace transform, the model space associated with $e^{-2s}$ corresponds to functions supported on the time interval $[0,2]$. This allows us to construct adjoint solutions whose boundary flux is asymptotically concentrated after any observation time $T<2$. A resolvent estimate then shows that their interior states at time $T$ remain bounded away from zero, yielding the failure of uniform observability. For negative speed, we use an explicit inverse Gram matrix to choose coefficients whose contribution have the same sign near $x=1$. This gives an interior-state lower bound that, compared with the boundary observation, yields the threshold $2 + 2 \sqrt{2}$.

The model-space viewpoint is also useful for the upper bounds. For both signs of $M$, it allows us to reduce finite-time observability to infinite-time observability whenever $T>2$. The infinite-time setting simplifies the formulas for the boundary-to-interior map, which expresses the interior state at time $T$ in terms of the full boundary observation.\footnote{This viewpoint was inspired by the work of Jacob and Partington~\cite{Jacob2006}, although we do not directly use their results.}~In particular, we derive an integral representation of the Hilbert--Schmidt norm of the boundary-to-interior map involving an auxiliary function $H_{\eps}$ defined by an infinite sine series~\eqref{eq:infinite_boundary_interior_kernel}. For negative speed, a direct estimate that neglects cancellation already yields the sharp upper bound $2+2\sqrt{2}$. For positive speed, localization estimates obtained by the Poisson summation formula and contour deformation, combined with the spatial weight in the adjoint eigenfunctions, yield the upper bound $2$.

Hardy spaces and model spaces therefore enter both the lower bound and the upper bound, but in different ways: for the positive-speed lower bound, they identify the limiting time-localization obstruction, whereas for the upper bounds they provide a convenient reduction to infinite-time observability. Hardy-space methods have appeared previously in the analysis of fast controls for the heat equation in~\cite{Lissy2026FastHeat}.

The remainder of the paper is organized as follows. Section~\ref{sec:framework} introduces the scaling, observability, and spectral framework. Section~\ref{sec:Hardy} develops the model-space tools. Sections~\ref{sec:proof} and~\ref{sec:negative_lower} establish the lower bounds for positive and negative speed, respectively. Section~\ref{sec:reduction_kernel} explains the reduction to infinite-time observability, introduces the boundary-to-interior map, derives its Hilbert--Schmidt norm formula, and proves the
localization estimate used for positive speed. Finally, Section~\ref{sec:proof_upper} proves the upper bounds and completes the proofs of the main theorems.

\section{Scaling, observability, and spectral structure}\label{sec:framework}\label{sec:strategy}

\subsection{Scaling and the observability formulation}\label{subsec:scaling_observability}
Throughout the rest of the paper, we assume $M \neq 0$. Note that by scaling time and space variables, we obtain the following scaling property:
\begin{equation}\label{eq:scaling_cost}
    C_{\eps}(T,L,M)
    =
    |M|^{-1/2}
    C_{\eps/(|M|L)}\left(\frac{T|M|}{L},1,\operatorname{sgn}M\right).
\end{equation}
Hence we may and shall assume $L=1$ and $M=\pm1$ in what follows. We write $\eta=\operatorname{sgn}M\in\{-1,1\}$ and
\begin{equation}\label{eq:normalized_cost}
    C_{\eps}^{\eta}(T)
    \coloneqq
    C_{\eps}(T,1,\eta).
\end{equation}

We next recall the standard duality between null-controllability and observability; see~\cite[Remark~2.98]{Coron2007}.

\begin{prop}\label{prop:duality}
    Let $z$ be the solution to
    \begin{equation}\label{eq:transport-diffusion-adjoint}
        \begin{dcases}
            z_t = \eta z_x + \eps z_{xx} & \text{in $(0,T) \times (0,1)$}, \\
            z(\cdot,0)=z(\cdot,1)=0 & \text{in $(0,T)$}, \\
            z(0,\cdot)=z_0 & \text{in $(0,1)$}
        \end{dcases}
    \end{equation}
    with $z_0 \in H_{0}^{1}(0,1) \cap H^2(0,1)$. Then we have
    \begin{equation}\label{eq:observation}
        \int_{0}^{1}|z(T,x)|^2 \, dx
        \leq
        C_{\eps}^{\eta}(T)^2
        \int_{0}^{T}|\eps z_x(t,0)|^2 \, dt.
    \end{equation}
    Conversely, $C_{\eps}^{\eta}(T)$ is the smallest nonnegative constant for which this inequality holds for all $z_0 \in H_{0}^{1}(0,1) \cap H^2(0,1)$.
\end{prop}

The adjoint equation is written here in forward time. We shall also consider its solutions for all $t>0$, in order to analyze infinite-time boundary observation.

\subsection{Spectral structure}\label{subsec:spectral_data}
As pointed out by Glass in~\cite{Glass2010}, the Dirichlet spectral data for the operator
\begin{equation}\label{eq:def_operator_P_signed}
    P_{\eps}^{\eta}
    \coloneqq
    -\left(\eta\partial_x+\eps\partial_x^2\right)
\end{equation}
with the domain $D(P_{\eps}^{\eta}) \coloneqq H_0^1(0,1)\cap H^2(0,1)$ can be computed explicitly. In fact, conjugation by $e^{-\eta x/(2\eps)}$ yields
\begin{equation}\label{eq:conjugation}
    e^{\eta x/(2\eps)} \circ P_{\eps}^{\eta} \circ e^{-\eta x/(2\eps)}
    =
    -\eps\partial_x^2+\frac{1}{4\eps}.
\end{equation}
Thus the eigenvalues, which do not depend on $\eta$, are
\begin{equation}\label{eq:lambda_n}
    \lambda_{n}^{\eps}
    \coloneqq
    \frac{1}{4\eps} + \eps n^2 \pi^2
    \qquad
    (n \geq 1)
\end{equation}
and corresponding eigenfunctions are
\begin{equation}\label{eq:phi_signed}
    \phi_{n}^{\eps,\eta}(x)
    \coloneqq
    \exp\left(-\frac{\eta x}{2\eps}\right)
    \sin(n\pi x)
    \qquad(n\geq1).
\end{equation}
Hence, for any natural number $M_{\eps}$ and any sequence of real numbers $(c_n^{\eps})_{n=1}^{M_{\eps}}$,
\begin{equation}\label{eq:finite_sum_solution}
    z^{\eps}
    =
    \sum_{n=1}^{M_{\eps}}
    c_n^{\eps} e^{-\lambda_{n}^{\eps} t} \phi_{n}^{\eps,\eta}
\end{equation}
is a solution to~\eqref{eq:transport-diffusion-adjoint} with the initial data $z_0^{\eps} = \sum_{n=1}^{M_{\eps}}c_n^{\eps} \phi_n^{\eps,\eta}$. Writing
\begin{equation}\label{eq:boundary_coefficients}
    b_n^{\eps}
    \coloneqq
    \eps\partial_x\phi_n^{\eps,\eta}(0)
    =
    \pi\eps n \qquad (n \geq 1),
\end{equation}
the boundary flux has the exponential-sum representation
\begin{equation}\label{eq:boundary_flux_spectral}
    \eps \partial_x z^{\eps}(t,0)
    =\sum_{n=1}^{M_{\eps}}b_n^{\eps} c_n^{\eps} e^{-\lambda_n^{\eps}t}.
\end{equation}
Note that the two signs $\eta = \pm 1$ give the same exponential family $e^{-\lambda_{n}^{\eps} t}$ and the same boundary coefficients $b_{n}^{\eps}$; the difference appears in the spatial weight $e^{-\eta x/(2\eps)}$ in the expression of the eigenfunctions.

\section{Model spaces and the vanishing viscosity limit}\label{sec:Hardy}

\subsection{Laplace transforms and the Hardy space $H^2$}
All functions appearing in this section are complex valued. We denote the right half-plane by
\begin{equation}
    \mathbb{C}_+
    \coloneqq
    \{ z \in \mathbb{C} : \Re z>0 \}.
\end{equation}
For $f \in L^2(0,\infty)$, we denote its Laplace transform by
\begin{equation}
    \mathcal{L}[f](s)
    \coloneqq
    \int_{0}^{\infty}e^{-st} f(t) \, dt \qquad (s \in \C_+).
\end{equation}
We recall two examples of Laplace transforms particularly important in the present study.

\begin{prop}\label{prop:Laplace_transform_exp}
    For $\lambda>0$, we have
    \begin{equation}
        \L[e^{-\lambda t}](s)
        =
        \frac{1}{\lambda+s} \qquad (s \in \C_+).
    \end{equation}
\end{prop}

\begin{prop}\label{prop:Laplace_transform_delay}
    For $f \in L^2(0,\infty)$ and $\tau>0$, we have
    \begin{equation}
        \L[ \bm{1}_{(\tau,\infty)} f(t-\tau) ]
        =
        e^{-\tau s} (\L f)(s) \qquad (s \in \C_+).
    \end{equation}
\end{prop}

Let us recall some required facts about the Hardy space $H^2=H^2(\mathbb{C}_+)$ over the right half-plane $\C_+$, following Duren~\cite{Duren1970} and Garcia and Ross~\cite{GarciaRoss2015}. Our notation $H^2$ corresponds to $\mathfrak{h}^2$ in the former and to $\mathcal{H}^2$ in the latter. Both references consider the upper half-plane instead of the right half-plane and the results should be transferred accordingly. We start with the definition of the Hardy space $H^2$.

\begin{dfn}
    Denote by $\mathcal{H}=\mathcal{H}(\C_+)$ the space of holomorphic functions on $\C_+$. For $F \in \mathcal{H}$, define its $H^2$-norm by
    \begin{equation}
        \| F \|_{H^2}
        \coloneqq
        \sup_{x>0}
        \left(
        \frac{1}{2\pi}
            \int_{-\infty}^{\infty}
            |F(x+iy)|^2 \, dy
        \right)^{1/2}.
    \end{equation}
    The Hardy space $H^2=H^2(\mathbb{C}_+)$ is the space
    \begin{equation}
        H^2
        \coloneqq
        \{ F \in \mathcal{H} : \| F \|_{H^2} < +\infty \}.
    \end{equation}
\end{dfn}

The next Proposition gives a fundamental property of the boundary behavior of functions in $H^2$.

\begin{prop}[{\cite[Section~10]{GarciaRoss2015}}]\label{prop:boundary_trace}
    For any $F \in H^2$, the limit
    \begin{equation}
        F_0(y)
        \coloneqq
        \lim_{x \to 0}
        F(x+iy)
    \end{equation}
    exists for almost every $y \in \R$. Moreover, we have $F_0 \in L^2(\R)$ and
    \begin{equation}
        \lim_{x \to 0}F(x+i\cdot)
        =
        F_0
        \qquad
        \text{in $L^2(\R)$}.
    \end{equation}
    In addition, the normalized restriction
    \begin{equation}
        \iota \colon H^2 \to L^2(\R)
    \end{equation}
    defined by $\iota(F)=(1/\sqrt{2\pi})F_0$ is an isometry and induces a Hilbert space structure on $H^2$ by the inner product
    \begin{equation}
        \langle F,G \rangle_{H^2}
        =
        \langle \iota(F),\iota(G) \rangle_{L^2(\R)}.
    \end{equation}
\end{prop}

We next recall the Paley--Wiener theorem.

\begin{prop}[{\cite[Section~11.5]{Duren1970}}]\label{prop:Laplace_isometry}
    The Laplace transform $\mathcal{L}$ is a unitary map 
    \begin{equation}
        \mathcal{L}
        \colon
        L^2(0,\infty)
        \to
        H^2.
    \end{equation}
    Moreover, for $F \in H^2$, we have
    \begin{equation}
        \L^{-1}[F](t)
        =
        \frac{1}{2\pi}
        \int_{-\infty}^{\infty}
        e^{iyt}F_0(y) \, dy,
    \end{equation}
    where $F_0$ is the boundary trace of $F$ defined in Proposition~\ref{prop:boundary_trace} and the integral is understood in the $L^2$ sense.
\end{prop}

An important property of the Hardy space $H^2$ is that it is a reproducing kernel Hilbert space.

\begin{prop}[{\cite[Section~10.3]{GarciaRoss2015}}]\label{prop:reproducing_kernel}
    For $\lambda>0$, the function
    \begin{equation}
        \frac{1}{\lambda+s}
    \end{equation}
    is a reproducing kernel of $H^2$ in the following sense:
    \begin{equation}
        F(\lambda)
        =
        \left\langle
            F,\frac{1}{\lambda+s}
        \right\rangle_{H^2}
        \qquad
        (F \in H^2).
    \end{equation}
\end{prop}

\begin{proof}
    By Propositions~\ref{prop:Laplace_transform_exp} and~\ref{prop:Laplace_isometry}, we have
    \begin{equation}
        \left\langle
            F,\frac{1}{\lambda+s}
        \right\rangle_{H^2}
        =
        \left\langle
            \mathcal{L}^{-1}[F],e^{-\lambda t}
        \right\rangle_{L^2(0,\infty)}
        =
        \mathcal{L}[\mathcal{L}^{-1}[F]](\lambda)
        =
        F(\lambda).
    \end{equation}
\end{proof}

\subsection{Inner functions and model spaces}
We next review some definitions and facts on inner functions on $\C_+$ and their model spaces.

\begin{dfn}
    $F \in \H(\C_+)$ is called an inner function on $\C_+$ if it is bounded in $\C_+$ and the limit
    \begin{equation}
        F_0(y)
        \coloneqq
        \lim_{x \to 0}F(x+iy)
    \end{equation}
    exists for almost every $y \in \R$ and
    \begin{equation}
        |F_0(y)| = 1
    \end{equation}
    for almost every $y \in \R$.
\end{dfn}

Important examples of inner functions are given by finite Blaschke products. For $0<\lambda_1<\lambda_2<\cdots<\lambda_N$, let
\begin{equation}\label{eq:Blaschke}
    \Theta(s)
    \coloneqq
    \prod_{n=1}^{N}
    \frac{\lambda_n-s}{\lambda_n+s}.
\end{equation}
The function $\Theta$ is an inner function on $\C_+$ and is called the finite Blaschke product with zeros $(\lambda_n)_{n=1}^{N}$.

Let $F$ be an inner function on $\C_+$. By definition, we have
\begin{equation}
    \| FG \|_{H^2}
    =
    \| G \|_{H^2} \qquad (G \in H^2).
\end{equation}
Thus the multiplication operator $M_F$ induced by $F$ defines an isometry
\begin{equation}
    M_F \colon H^2 \to H^2.
\end{equation}
This implies that $F H^2$ is a closed subspace of $H^2$, which allows us to take its orthogonal complement in $H^2$.

\begin{dfn}
    Let $F$ be an inner function on $\C_+$. We denote the orthogonal complement of $F H^2$ in $H^2$ by
    \begin{equation}
        K_{F}
        \coloneqq
        H^2 \ominus (F H^2).
    \end{equation}
    The space $K_F$ is called the model space of $F$.
\end{dfn}

We give two examples of model spaces important for our purpose.

\begin{prop}\label{prop:model_exp}
    For $\tau>0$, we have
    \begin{equation}
        K_{e^{-\tau s}}
        =
        \L
        \left(
            L^2(0,\tau)
        \right).
    \end{equation}
    Here, $L^2(0,\tau)$ is naturally embedded into $L^2(0,\infty)$.
\end{prop}

\begin{proof}
    By Propositions~\ref{prop:Laplace_transform_delay} and~\ref{prop:Laplace_isometry}, we have
    \begin{equation}
        e^{-\tau s} H^2
        =
        \L \left( L^2(\tau,\infty) \right).
    \end{equation}
    Since
    \begin{equation}
        L^2(0,\tau)
        =
        L^2(0,\infty)
        \ominus
        L^2(\tau,\infty),
    \end{equation}
    we conclude that
    \begin{equation}
        \L
        \left(
            L^2(0,\tau)
        \right)
        =
        H^2
        \ominus
        e^{-\tau s} H^2
        =
        K_{e^{-\tau s}}.
    \end{equation}
\end{proof}

\begin{prop}\label{prop:model_Theta}
    Let $\Theta$ be the inner function defined by~\eqref{eq:Blaschke}. Then
    \begin{equation}\label{eq:model_Theta}
        K_{\Theta}
        =
        \Span
        \left\{
            \frac{1}{\lambda_1+s}, \ldots, \frac{1}{\lambda_N+s}
        \right\}
        =
        \mathcal{L}
        \left(
        \Span
        \left\{
            e^{-\lambda_1 t}, \ldots, e^{-\lambda_N t}
        \right\}
        \right).
    \end{equation}
\end{prop}

\begin{proof}
    For $F \in H^2$, Proposition~\ref{prop:reproducing_kernel} implies
    \begin{equation}
        \left\langle
            \Theta F, \frac{1}{\lambda_n+s}
        \right\rangle_{H^2}
        =
        \Theta(\lambda_n) F(\lambda_n)
        =
        0
    \end{equation}
    for $1 \leq n \leq N$. Hence
    \begin{equation}\label{eq:subset}
        \Theta H^2
        \subset
        \Span
        \left\{
            \frac{1}{\lambda_1+s}, \ldots, \frac{1}{\lambda_N+s}
        \right\}^{\perp}.
    \end{equation}
    On the other hand, if
    \begin{equation}
        F
        \in 
        \Span
        \left\{
            \frac{1}{\lambda_1+s}, \ldots, \frac{1}{\lambda_N+s}
        \right\}^{\perp},
    \end{equation}
    then $F(\lambda_n)=\langle F,1/(\lambda_n+s) \rangle_{H^2}=0$ for $1 \leq n \leq N$. Hence $F/\Theta \in H^2$. This implies
    \begin{equation}\label{eq:supset}
        \Theta H^2
        \supset
        \Span
        \left\{
            \frac{1}{\lambda_1+s}, \ldots, \frac{1}{\lambda_N+s}
        \right\}^{\perp}.
    \end{equation}
    Combining~\eqref{eq:subset} and~\eqref{eq:supset}, we obtain the first equality in~\eqref{eq:model_Theta}; the second equality follows from Proposition~\ref{prop:Laplace_transform_exp}.
\end{proof}

Propositions~\ref{prop:model_exp} and~\ref{prop:model_Theta} explain the role of model spaces in our analysis. Under the inverse Laplace transform, the model space $K_{\Theta}$ is precisely the space of temporal profiles generated by the exponential family $(e^{-\lambda_n t})_{n=1}^{N}$, whereas the model space $K_{e^{-\tau s}}$ corresponds to $L^2(0,\tau)$. In particular, if $\lambda_n = \lambda_{n}^{\eps}$ is the Dirichlet eigenvalue~\eqref{eq:lambda_n}, the model space $K_{\Theta}$ describes possible profiles of the boundary flux~\eqref{eq:boundary_flux_spectral}. Hence the structure of the boundary flux in the vanishing viscosity limit can be understood by analyzing the corresponding limit of the model spaces. The relevance of the exponential inner function $e^{-\tau s}$ will become clear in the next two subsections. With a suitable choice of the truncation index $N_{\eps}$, the Blaschke product $\Theta_{\eps}$ associated with $(\lambda_{n}^{\eps})_{n=1}^{N_{\eps}}$ converges to $e^{-2s}$. Correspondingly, the model space $K_{\Theta_{\eps}}$ converges in the sense of projections to the Laplace transform of $L^2(0,2)$.

\subsection{The vanishing viscosity limit of the Blaschke products}
We next compute the vanishing viscosity limit of the finite Blaschke products associated with the eigenvalues $(\lambda_{n}^{\eps})_{n=1}^{\infty}$.

\begin{prop}\label{prop:convergence_Blaschke_products}
    Let $\lambda_{n}^{\eps}$ be defined by~\eqref{eq:lambda_n}. Set $N_{\eps} = \lfloor 1/\eps^2 \rfloor$ and
    \begin{equation}\label{eq:Blaschke_eps}
        \Theta_{\eps}(s)
        \coloneqq
        \prod_{n=1}^{N_{\eps}}
        \frac{\lambda_{n}^{\eps}-s}{\lambda_{n}^{\eps}+s}.
    \end{equation}
    Then we have
    \begin{equation}
        \Theta_{\eps}(iy)
        \to
        e^{-2iy}
    \end{equation}
    for all $y \in \R$.
\end{prop}

\begin{proof}
    Fix $y \in \R$. For $\lambda>0$, we have the formula
    \begin{equation}
        \frac{\lambda-iy}{\lambda+iy}
        =
        \exp
        \left(
            -2i \arctan \left( \frac{y}{\lambda} \right)
        \right).
    \end{equation}
    See Figure~\ref{fig:blaschke-angle}. Hence
    \begin{equation}\label{eq:theta_boundary_arctan}
        \Theta_{\eps}(iy)
        =
        \exp
        \left(
            -2i \sum_{n=1}^{N_{\eps}}\arctan \left( \frac{y}{\lambda_{n}^{\eps}} \right)
        \right).
    \end{equation}
    Let
    \begin{equation}\label{eq:deps}
        d_{\eps}
        \coloneqq
        \sum_{n=1}^{N_{\eps}}
        \frac{1}{\lambda_{n}^{\eps}}
        =
        \eps
        \sum_{n=1}^{N_{\eps}}
        \frac{1}{1/4+\pi^2 (\eps n)^2}.
    \end{equation}
    Since
    \begin{equation}\label{eq:deps_lower_upper}
        \int_{\eps}^{\eps(N_{\eps}+1)}\frac{dx}{1/4+\pi^2 x^2}
        \leq
        d_{\eps}
        \leq
        \int_{0}^{\eps N_{\eps}}\frac{dx}{1/4+\pi^2 x^2},
    \end{equation}
    it follows that
    \begin{equation}\label{eq:theta_boundary_d}
        \lim_{\eps \to 0}d_{\eps}
        =
        \int_{0}^{\infty}\frac{dx}{1/4+\pi^2 x^2}
        =
        1.
    \end{equation}
    On the other hand, we have
    \begin{equation}\label{eq:theta_boundary_e}
        e_{\eps}
        \coloneqq
        \sum_{n=1}^{N_{\eps}}
        \frac{1}{(\lambda_{n}^{\eps})^3}
        \leq
        \eps^3
        \sum_{n=1}^{N_{\eps}}
        \frac{1}{(1/4)^3}
        \leq 64\eps
        \to
        0
        \qquad
        \text{as $\eps \to 0$.}
    \end{equation}
    Since
    \begin{equation}
        \max_{1 \leq n \leq N_{\eps}}
        \frac{|y|}{\lambda_{n}^{\eps}}
        \leq
        4 |y| \eps
    \end{equation}
    and
    \begin{equation}
        \arctan x
        =
        x
        +
        O(|x|^3)
        \qquad
        \text{as $x \to 0$},
    \end{equation}
    we conclude from~\eqref{eq:theta_boundary_arctan},~\eqref{eq:theta_boundary_d}, and~\eqref{eq:theta_boundary_e} that
    \begin{equation}
        \lim_{\eps \to 0}\Theta_{\eps}(iy)
        =
        \exp(-2iy).
    \end{equation}
    This ends the proof.
\end{proof}

\begin{figure}[htbp]
    \centering
    \includegraphics[width=0.5\textwidth]{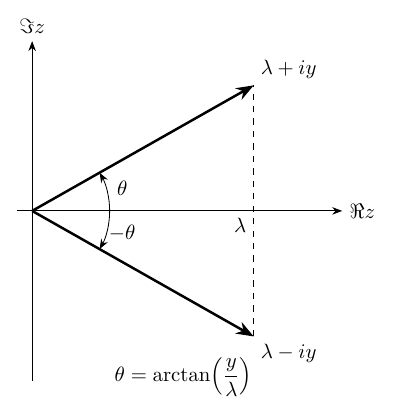}
    \caption{The configuration of $\lambda \pm iy$ in the complex plane.}
    \label{fig:blaschke-angle}
\end{figure}

\begin{remark}
    We can also show that $\Theta_{\eps}(s) \to e^{-2s}$ uniformly on any compact subset of $\C_+$ by using~\eqref{eq:theta_boundary_d} and~\eqref{eq:theta_boundary_e}.
    This also follows from Proposition~\ref{prop:convergence_Blaschke_products} and the Poisson representation formula for bounded Hardy functions on $\C_+$; see~\cite[Theorem~11.2]{Duren1970}.
\end{remark}

\subsection{Convergence of model-space projections}
The convergence of the finite Blaschke products proved in the previous section can be transferred to the convergence of the projections onto the corresponding model spaces. We prove this in a general setting.

\begin{prop}\label{prop:convergence_model_spaces}
    Let $(F_n)_{n=1}^{\infty}$ and $F$ be inner functions on $\C_+$. If
    \begin{equation}\label{eq:convergence_inner_functions}
        \lim_{n \to \infty}F_n(iy)
        =
        F(iy)
    \end{equation}
    for almost every $y \in \R$, then
    \begin{equation}\label{eq:convergence_projections_model_spaces}
        \lim_{n \to \infty}
        P_{K_{F_n}} G
        =
        P_{K_F} G
        \qquad
        \text{in $H^2$}
    \end{equation}
    for all $G \in H^2$.
\end{prop}

\begin{proof}
    We first note that the following identity holds for any Hilbert space $H$, its closed subspace $X$, and a unitary operator $U \colon H \to H$:
    \begin{equation}\label{eq:projection_identity}
        P_{UX}
        =
        U
        P_X
        U^*,
    \end{equation}
    where $P_X$ and $P_{UX}$ are the orthogonal projections to $X$ and $UX$, respectively.
    
    In this proof, by virtue of Proposition~\ref{prop:boundary_trace}, we identify $H^2$ as a closed subspace of $L^2(\R)$. Since the multiplication operators $M_{F_n} \colon L^2(\R) \to L^2(\R)$ and $M_F \colon L^2(\R) \to L^2(\R)$ are unitary,~\eqref{eq:projection_identity} implies
    \begin{equation}
        P_{F_n H^2}
        =
        M_{F_n} P_{H^2} M_{\overline{F_n}},
        \qquad
        P_{F H^2}
        =
        M_{F} P_{H^2} M_{\overline{F}}.
    \end{equation}
    Hence
    \begin{equation}\label{eq:projection_model_space_decomposition}
        P_{K_{F_n}}
        =
        I
        -
        M_{F_n} P_{H^2} M_{\overline{F_n}},
        \qquad
        P_{K_F}
        =
        I
        -
        M_{F} P_{H^2} M_{\overline{F}}.
    \end{equation}
    On the other hand, by~\eqref{eq:convergence_inner_functions}, $|F_n|=|F|=1$ on $i\R = \partial \C_+$, so Lebesgue's dominated convergence theorem gives
    \begin{equation}
        \lim_{n \to \infty}
        M_{F_n} G
        =
        M_F G,
        \qquad
        \lim_{n \to \infty}
        M_{\overline{F_n}} G
        =
        M_{\overline{F}} G
        \qquad
        \text{in $L^2(\R)$},
    \end{equation}
    for all $G \in L^2(\R)$. Therefore, for any $G \in H^2$,
    \begin{align}
        \begin{aligned}
            &
            \|
                M_{F_n} P_{H^2} M_{\overline{F_n}} G
                -
                M_{F} P_{H^2} M_{\overline{F}} G
            \|_{H^2}
            \\
            & \leq
            \|
                M_{F_n} P_{H^2}
                \left(
                    M_{\overline{F_n}}
                    -
                    M_{\overline{F}}
                \right)
                G
            \|_{H^2}
            +
            \|
                (M_{F_n}-M_{F}) P_{H^2} M_{\overline{F}} G
            \|_{H^2}
            \\
            & \to 0
        \end{aligned}
    \end{align}
    as $n \to \infty$. Combining this with~\eqref{eq:projection_model_space_decomposition} yields~\eqref{eq:convergence_projections_model_spaces}.
\end{proof}

Combining Propositions~\ref{prop:convergence_model_spaces} and~\ref{prop:convergence_Blaschke_products}, we obtain the following.

\begin{prop}\label{prop:convergence_projections_time_exponentials}
    Let
    \begin{equation}
        \mathcal{E}_{\eps}
        \coloneqq
        \Span
        \left\{
            e^{-\lambda_{n}^{\eps} t} : 1 \leq n \leq N_{\eps}
        \right\}
        \subset
        L^2(0,\infty)
    \end{equation}
    and denote by $\Pi_{\eps} \colon L^2(0,\infty) \to L^2(0,\infty)$ the orthogonal projection to $\mathcal{E}_{\eps}$. Then
    \begin{equation}
        \lim_{\eps \to 0}
        \Pi_{\eps} f
        =
        \bm{1}_{(0,2)} f
        \qquad
        \text{in $L^2(0,\infty)$}
    \end{equation}
    for all $f \in L^2(0,\infty)$.
\end{prop}

\begin{proof}
    By Propositions~\ref{prop:Laplace_transform_exp},~\ref{prop:Laplace_isometry}, and~\ref{prop:model_Theta}, we have
    \begin{equation}
        \mathcal{E}_{\eps}
        =
        \L^{-1} \left( K_{\Theta_{\eps}} \right),
    \end{equation}
    where $K_{\Theta_{\eps}}$ is the model space of the Blaschke product $\Theta_{\eps}$ defined by~\eqref{eq:Blaschke_eps}. Similarly, Proposition~\ref{prop:model_exp} implies
    \begin{equation}
        L^2(0,2)
        =
        \L^{-1} \left( K_{e^{-2s}} \right).
    \end{equation}
    Then the conclusion follows from Propositions~\ref{prop:Laplace_isometry},~\ref{prop:convergence_model_spaces}, and~\ref{prop:convergence_Blaschke_products}.
\end{proof}

Proposition~\ref{prop:convergence_projections_time_exponentials} shows that every temporal profile $g \in L^2(0,2) \subset L^2(0,\infty)$ can be approximated in $L^2(0,\infty)$ by boundary fluxes of the form~\eqref{eq:boundary_flux_spectral} with $M_{\eps}=N_{\eps}$. This approximation property plays an essential role in the proof of~\eqref{thm:main_positive_blowup} in Theorem~\ref{thm:main_positive}.

\section{Failure of uniform observability below time $2$}\label{sec:proof}
In this section, we prove the blow-up implication~\eqref{thm:main_positive_blowup} in Theorem~\ref{thm:main_positive}. Throughout this section, we take $\eta=1$ and write
\begin{equation}
    P_{\eps} \coloneqq P_{\eps}^{1} = -(\partial_x + \eps \partial_{x}^{2}), \qquad \varphi_{n}^{\eps} \coloneqq \varphi_{n}^{\eps,1} =
    \exp\left( -\frac{x}{2\eps} \right) \sin(n \pi x).
\end{equation}
By Proposition~\ref{prop:duality}, it suffices to show the following.

\begin{thm}\label{thm:counterexample_observation}
    For any $0<T<2$, there exists a positive constant $c>0$ and a family of solutions $(z^{\eps})_{\eps>0}$ to~\eqref{eq:transport-diffusion-adjoint} with $z^{\eps}(0,\cdot) \in H_{0}^{1}(0,1) \cap H^2(0,1)$ such that
    \begin{equation}\label{eq:z_lower_bound}
        \int_{0}^{1}|z^{\eps}(T,x)|^2 \, dx \geq c
    \end{equation}
    for sufficiently small $\eps>0$ and
    \begin{equation}\label{eq:z_vanishing_flux}
        \lim_{\eps \to 0}
        \int_{0}^{T}|\eps z_{x}^{\eps}(t,0)|^2 \, dt
        =
        0.
    \end{equation}
\end{thm}

To prove Theorem~\ref{thm:counterexample_observation}, we first need some preparations. Assume $0<T<2$ and take $\alpha<\beta$ such that
\begin{equation}
    T<\alpha<\beta<2.
\end{equation}
We may take $g \in C^{\infty}([0,\infty))$ such that
\begin{equation}
    \supp g \subset [\alpha,\beta], \qquad g \geq 0, \qquad g \not \equiv 0.
\end{equation}
Then, using Proposition~\ref{prop:convergence_projections_time_exponentials}, we can construct boundary fluxes involving finitely many modes that converge to $g$ in the limit $\eps \to 0$.

\begin{prop}\label{prop:q}
    Let
    \begin{equation}
        q_{\eps}
        \coloneqq
        \Pi_{\eps}
        g \in L^2(0,\infty),
    \end{equation}
    where $\Pi_{\eps}$ is defined in Proposition~\ref{prop:convergence_projections_time_exponentials}. We have
    \begin{equation}\label{eq:limit_q_g}
        \lim_{\eps \to 0}
        q_{\eps}
        =g
        \qquad
        \text{in $L^2(0,\infty)$}
    \end{equation}
    and
    \begin{equation}\label{eq:limit_q_boundary}
        \lim_{\eps \to 0}
        \int_{0}^{T}
        |q_{\eps}(t)|^2
        \, dt
        =0.
    \end{equation}
    Moreover, there exist real numbers $(a_{n}^{\eps})_{n=1}^{N_{\eps}}$ such that
    \begin{equation}\label{eq:q_real_coefficients}
        q_{\eps}(t)
        =
        \sum_{n=1}^{N_{\eps}}
        a_{n}^{\eps} e^{-\lambda_{n}^{\eps} t}
        \qquad
        (t \geq 0).
    \end{equation}
\end{prop}

\begin{proof}
    By Proposition~\ref{prop:convergence_projections_time_exponentials}, we have
    \begin{equation}
        \lim_{\eps \to 0}
        q_{\eps}
        =
        \bm{1}_{(0,2)} g
        =
        g
        \qquad
        \text{in $L^2(0,\infty)$},
    \end{equation}
    which proves~\eqref{eq:limit_q_g}. On the other hand, since
    \begin{align}
        \begin{aligned}
            \| q_{\eps} \|_{L^2(0,T)}
            & = \| g + (q_{\eps} - g) \|_{L^2(0,T)} \\
            & \leq
            \| g \|_{L^2(0,T)}
            +
            \| q_{\eps} - g \|_{L^2(0,T)} \\
            & =
            \| q_{\eps} - g \|_{L^2(0,T)},
        \end{aligned}
    \end{align}
    the equality~\eqref{eq:limit_q_boundary} follows from~\eqref{eq:limit_q_g}. Finally,~\eqref{eq:q_real_coefficients} follows from the definition of the space $\mathcal{E}_{\eps}$ and the fact that both $e^{-\lambda_{n}^{\eps} t}$ and $g$ are real functions.
\end{proof}

The preceding construction makes the observed flux small on $(0,T)$. To obtain a counterexample to observability, we must also keep the interior state at time $T$ away from zero. The next lemma provides the required lower bound in terms of the future boundary flux.


\begin{lemma}\label{lemma:lower_bound_present_from_future_flux}
    For real numbers $(a_n)_{n=1}^{M}$, consider $w_0 \colon [0,1] \to \R$ given by
    \begin{equation}\label{eq:w0}
        w_0
        \coloneqq
        \sum_{n=1}^{M}
        a_n \phi_{n}^{\eps}.
    \end{equation}
    Let $w \colon [0,\infty) \times [0,1] \to \R$ be the function defined by
    \begin{equation}\label{eq:def_w}
        w
        \coloneqq
        e^{-t P_{\eps}} w_0
        =
        \sum_{n=1}^{M}
        a_n e^{-\lambda_{n}^{\eps} t}\phi_{n}^{\eps}.
    \end{equation}
    Then for any $\sigma>0$, we have
    \begin{equation}\label{eq:lower_bound_present_from_future_flux}
        \left|
            \int_{0}^{\infty}
            e^{-\sigma t}
            \eps
            \partial_x w(t,0)
            \, dt
        \right|
        \leq
        \| w_0 \|_{L^2(0,1)}.
    \end{equation}
\end{lemma}

\begin{proof}
    Define $u \colon [0,1] \to \R$ by
    \begin{equation}
        u
        \coloneqq
        (\sigma + P_{\eps})^{-1} w_0
        =
        \int_{0}^{\infty}
        e^{-\sigma t}
        w
        \, dt.
    \end{equation}
    The second equality is a representation of the resolvent $(\sigma+P_{\eps})^{-1}$ in terms of the Laplace transform, which can also be directly verified for $w_0$ having the form~\eqref{eq:w0}. We note that $u$ solves the elliptic equation
    \begin{equation}
        \begin{dcases}
            (\sigma - \partial_x - \eps \partial_{x}^{2})
            u
            =
            w_0
            & \text{in $(0,1)$}, \\
            u(0)=u(1)=0.
        \end{dcases}
    \end{equation}
    To estimate the boundary derivative $\eps u_x(0)$, consider the adjoint boundary-value problem
    \begin{equation}
        \begin{dcases}
            (\sigma + \partial_x - \eps \partial_{x}^{2})
            v
            =
            0
            & \text{in $(0,1)$}, \\
            v(0)=1, \qquad v(1)=0.
        \end{dcases}
    \end{equation}
    Then by integration by parts, we have
    \begin{equation}
        \int_{0}^{1} w_0 v \, dx
        =
        \int_{0}^{1}
        \left\{
            (\sigma - \partial_x - \eps \partial_{x}^{2}) u
        \right\}
        v \, dx
        =
        \eps \partial_x u(0).
    \end{equation}
    By the maximum principle, we have $0 \leq v \leq 1$. Hence
    \begin{equation}
        |\eps \partial_x u(0)|
        \leq
        \int_{0}^{1} |w_0| \, dx
        \leq
        \| w_0 \|_{L^2(0,1)}.
    \end{equation}
    This proves~\eqref{eq:lower_bound_present_from_future_flux}; differentiation under the integral sign can be verified easily as $w$ is defined by the finite sum~\eqref{eq:def_w}.
\end{proof}

We now prove Theorem~\ref{thm:counterexample_observation}.

\begin{proof}[Proof of Theorem~\ref{thm:counterexample_observation}]
    Using the coefficients $(a_{n}^{\eps})_{n=1}^{N_{\eps}}$ in Proposition~\ref{prop:q}, we define
    \begin{equation}
        c_{n}^{\eps}
        \coloneqq
        \frac{a_{n}^{\eps}}{
        \pi \eps n}
        \qquad (1 \leq n \leq N_{\eps}).
    \end{equation}
    Then define $z^{\eps} \colon [0,\infty) \times [0,1] \to \R$ by
    \begin{equation}
        z^{\eps}
        \coloneqq
        \sum_{n=1}^{N_{\eps}}
        c_{n}^{\eps} e^{-\lambda_{n}^{\eps} t} \phi_{n}^{\eps}.
    \end{equation}
    As noted in Section~\ref{subsec:spectral_data}, the function $z^{\eps}$ is a solution to~\eqref{eq:transport-diffusion-adjoint} with the initial data $z_0^{\eps} = \sum_{n=1}^{N_{\eps}}c_n^{\eps} \phi_{n}^{\eps}$. Of course, $z_0^{\eps} \in H_{0}^{1}(0,1) \cap H^2(0,1)$. By definition, we have
    \begin{equation}
        \eps z_{x}^{\eps}(t,0)
        =
        q_{\eps}(t),
    \end{equation}
    where $q_{\eps}$ is defined in Proposition~\ref{prop:q}. Thus~\eqref{eq:limit_q_boundary} implies~\eqref{eq:z_vanishing_flux}.
    
    It remains to prove~\eqref{eq:z_lower_bound}. To this end, let 
    \begin{equation}\label{eq:w0_zT}
        w_0
        \coloneqq
        z^{\eps}(T,\cdot).
    \end{equation}
    Then we have
    \begin{equation}
        w
        \coloneqq
        e^{-t P_{\eps}}
        w_0
        =
        z^{\eps}(T+t,\cdot).
    \end{equation}
    Fix $\sigma>0$. Applying Lemma~\ref{lemma:lower_bound_present_from_future_flux} to $w_0$ defined by~\eqref{eq:w0_zT}, we obtain
    \begin{equation}\label{eq:z_lower_bound_pre}
        \| z^{\eps}(T,\cdot) \|_{L^2(0,1)}
        \geq
        \left|
            \int_{0}^{\infty}
            e^{-\sigma t}
            \eps
            \partial_x z^{\eps}(T+t,0)
            \, dt
        \right|
        =
        \left|
            \int_{0}^{\infty}
            e^{-\sigma t}
            q_{\eps}(T+t)
            \, dt
        \right|.
    \end{equation}
    By~\eqref{eq:limit_q_g}, the right-hand side converges as $\eps \to 0$ to
    \begin{equation}\label{eq:limiting_flux}
        \left|
            \int_{0}^{\infty}
            e^{-\sigma t}
            g(T+t)
            \, dt
        \right|
        >0.
    \end{equation}
    Combining~\eqref{eq:z_lower_bound_pre} and~\eqref{eq:limiting_flux} implies the existence of a positive constant $c>0$ such that~\eqref{eq:z_lower_bound} holds for sufficiently small $\eps>0$.
\end{proof}

We have thus proved Theorem~\ref{thm:counterexample_observation}. Combining it with Proposition~\ref{prop:duality} and the scaling identity~\eqref{eq:scaling_cost} proves the blow-up implication~\eqref{thm:main_positive_blowup} in Theorem~\ref{thm:main_positive}.

\section{Failure of uniform observability below time $2+2\sqrt{2}$}\label{sec:negative_lower}
In this section, we prove the blow-up implication~\eqref{thm:main_negative_blowup} in Theorem~\ref{thm:main_negative}. Throughout this section, we take $\eta=-1$ and write
\begin{equation}
    \varphi_{n}^{\eps}
    \coloneqq
    \varphi_{n}^{\eps,-1}
    =
    \exp
    \left(
        \frac{x}{2 \eps}
    \right)
    \sin(n \pi x).
\end{equation}

As in the previous sections, we set $N_{\eps} = \lfloor 1 / \eps^2 \rfloor$. Let $\bm{c} = (c_n)_{n=1}^{N_{\eps}}$ be real numbers and consider
\begin{equation}
    z_0^{\eps}(x)
    =
    \sum_{n=1}^{N_{\eps}}
    c_n \varphi_{n}^{\eps}(x).
\end{equation}
The solution to~\eqref{eq:transport-diffusion-adjoint} with the initial data $z_{0}^{\eps}$ is then
\begin{equation}
    z^{\eps}(t,x)
    =
    \sum_{n=1}^{N_{\eps}}
    c_n e^{-\lambda_{n}^{\eps} t} \varphi_{n}^{\eps}(x).
\end{equation}
Its scaled boundary flux is given by
\begin{equation}\label{eq:q_negative_counterexample}
    \eps \partial_x z^{\eps}(t,0)
    =
    \sum_{n=1}^{N_{\eps}}
    b_{n}^{\eps} c_n e^{-\lambda_{n}^{\eps} t},
\end{equation}
where $b_{n}^{\eps}$ is defined by~\eqref{eq:boundary_coefficients}. We start with the following simple observation.

\begin{prop}\label{prop:Gram}
    Define the Gram matrix $G=G^{\eps} \in \R^{N_{\eps} \times N_{\eps}}$ by
    \begin{equation}
        G_{mn}
        \coloneqq
        \int_{0}^{\infty}
        b_m^{\eps}
        b_n^{\eps}
        e^{-\lambda_{m}^{\eps} t}
        e^{-\lambda_{n}^{\eps} t}
        \, dt
        =
        \frac{b_m^{\eps} b_n^{\eps}}{\lambda_{m}^{\eps} + \lambda_{n}^{\eps}}
        \qquad
        (1 \leq m, n \leq N_{\eps}).
    \end{equation}
    Then we have
    \begin{equation}\label{eq:negative_flux_gram_norm}
        \int_{0}^{\infty}
        |\eps \partial_x z^{\eps}(t,0)|^2
        \, dt
        =
        \bm{c}^T
        G
        \bm{c}.
    \end{equation}
    Moreover, the inverse matrix of $G$ is given by
    \begin{equation}\label{eq:negative_flux_gram_inverse}
        \left(
            G^{-1}
        \right)_{mn}
        =
        \frac{
        4(-1)^{m+n}
        \lambda_{m}^{\eps} \lambda_{n}^{\eps}
        }
        {b_m^{\eps} b_n^{\eps} (\lambda_{m}^{\eps} + \lambda_{n}^{\eps}) d_{m,N_{\eps}}^{\eps} d_{n,N_{\eps}}^{\eps}
        }
        \qquad
        (1 \leq m, n \leq N_{\eps}),
    \end{equation}
    where
    \begin{equation}\label{eq:finite_product_d_firstdef}
        d_{n,N}^{\eps}
        \coloneqq
        \prod_{\substack{1\leq k\leq N \\ k\ne n}}
        \frac{|\lambda_k^{\eps}-\lambda_n^{\eps}|}
        {\lambda_k^{\eps}+\lambda_n^{\eps}}
        \qquad
        (N \in \mathbb{N}, \ 1 \leq n \leq N).
    \end{equation}
\end{prop}

\begin{proof}
    Equality~\eqref{eq:negative_flux_gram_norm} is obvious from the definition of the matrix $G$. Abbreviate $N=N_\eps$, $\lambda_n=\lambda_n^\eps$, $b_n=b_n^\eps$, and $d_n=d_{n,N_{\eps}}^\eps$. With
    \begin{equation}
        D
        \coloneqq
        \operatorname{diag}(b_1,\ldots,b_N),
        \qquad
        A
        \coloneqq
        \left(
            \frac{1}{\lambda_m+\lambda_n}
        \right)_{1 \leq m, n \leq N},
    \end{equation}
    we can write $G$ in the form
    \begin{equation}
        G
        =
        D A D.
    \end{equation}
    The matrix $A$ is a Cauchy matrix, and by the Cauchy matrix inversion formula~\cite{Schechter1959}, we obtain
    \begin{equation}
        (A^{-1})_{mn}
        =
        \frac{4\lambda_m \lambda_n p_m p_n}{\lambda_m + \lambda_n};
        \qquad
        p_n
        =
        \prod_{\substack{1\leq k\leq N_{\eps} \\ k\ne n}}
        \frac{
            \lambda_n + \lambda_k
        }
        {
            \lambda_n-\lambda_k
        }.
    \end{equation}
    Note that since $\lambda_1<\cdots<\lambda_N$, we have
    \begin{equation}
        p_n
        =
        \frac{(-1)^{N-n}}{d_n}.
    \end{equation}
    Substituting this into the preceding formula and using $G^{-1}=D^{-1}A^{-1}D^{-1}$ proves~\eqref{eq:negative_flux_gram_inverse}.
\end{proof}

We next compute the infinite-product limit of $d_{n,N}^{\eps}$ as $N \to \infty$.

\begin{lemma}\label{lem:d_infinite}
    Let $d_{n,N}^{\eps}$ be the product defined by~\eqref{eq:finite_product_d_firstdef} and let
    \begin{equation}
        d_{n}^{\eps}
        \coloneqq
        \lim_{N \to \infty}
        d_{n,N}^{\eps}
        =
        \prod_{\substack{k \geq1 \\ k \ne n}}
        \frac{| \lambda_{k}^{\eps} - \lambda_{n}^{\eps} |}{\lambda_{k}^{\eps} + \lambda_{n}^{\eps}}.
    \end{equation}
    Then
    \begin{equation}\label{eq:infinite_product_d}
        d_{n,N}^{\eps} 
        \geq
        d_{n}^{\eps}
        =
        \frac{\lambda_{n}^{\eps} \omega_{n}^{\eps}}{\eps \pi^2 n^2 \sinh(\omega_n^\eps)},
    \end{equation}
    where
    \begin{equation}\label{eq:omega_boundary_interior_firstdef}
        \omega_n^{\eps}
        \coloneqq
        \sqrt{\pi^2 n^2 + \frac{1}{2\eps^2}}.
    \end{equation}
\end{lemma}

\begin{proof}
    The inequality in~\eqref{eq:infinite_product_d} follows easily noting that each factor in the product has modulus less than $1$. The infinite product representations for $\sin$ and $\sinh$ give
    \begin{equation}
        \prod_{\substack{k\geq1\\ k\ne n}}
        \left|
            1-\frac{n^2}{k^2}
        \right|
        =
        \frac{1}{2},
        \qquad
        \prod_{k=1}^{\infty}
        \left\{
            1 + \frac{(\omega_n^{\eps})^2}{\pi^2 k^2}
        \right\}
        =
        \frac{\sinh(\omega_n^{\eps})}{\omega_n^{\eps}}.
    \end{equation}
    Dividing the former by the latter and removing the $k=n$ factor from the denominator yields the equality in~\eqref{eq:infinite_product_d}.
\end{proof}

The previous lemma, together with an estimate of the truncation error, gives an upper bound on $d_{1,N_{\eps}}^{\eps}$ needed in the construction of adjoint solutions exhibiting the failure of uniform observability.

\begin{lemma}\label{lem:first_product_bound}
    Let $N_{\eps} = \lfloor 1 / \eps^2 \rfloor$. There exists a constant $C>0$ such that
    \begin{equation}\label{eq:first_product_bound}
        d_{1,N_\eps}^{\eps}
        \leq
        C
        \eps^{-3}
        \exp
        \left(
            -\frac{1}{\sqrt{2}\eps}
        \right)
    \end{equation}
    for $0 < \eps \leq 1 / 2$.
\end{lemma}

\begin{proof}
    First, note that
    \begin{equation}
        \frac{d_{1,N_{\eps}}^{\eps}}
        {d_{1}^{\eps}}
        =
        \prod_{k>N_{\eps}}
        \frac{\lambda_{k}^{\eps} + \lambda_{1}^{\eps}}{|\lambda_{k}^{\eps} - \lambda_{1}^{\eps}|}
        =
        \prod_{k>N_{\eps}}       
        \left(
            1 + \frac{2 + 1 / (2 \pi^2 \eps^2)}{k^2-1}
        \right).
    \end{equation}
    Since $N_{\eps} \geq 1 / (2 \eps^2)$, applying the inequality $\log(1+x) \leq x$ and computing a telescopic sum, we obtain\footnote{We remark that a similar calculation can be found in~\cite{Jacob2006}.}
    \begin{align}\label{eq:first_product_tail}
        \begin{aligned}
            \log \frac{d_{1,N_{\eps}}^{\eps}}
            {d_{1}^{\eps}}
            &
            \leq
            \left(
                2 + \frac{1}{2 \pi^2 \eps^2}
            \right)
            \sum_{k>N_{\eps}}
            \frac{1}{k^2-1}
            \\
            &
            \leq
            \left(
                2 + \frac{1}{2 \pi^2 \eps^2}
            \right)
            \frac{1}{N_{\eps}}
            \\
            &
            \leq
            1 + \frac{1}{\pi^2}.
        \end{aligned}
    \end{align}
    For $0 < \eps \leq 1 / 2$, there exists $C>0$ such that
    \begin{equation*}
        \lambda_1^\eps
        \leq
        C
        \eps^{-1},
        \qquad
        \frac{1}{\sqrt{2}\eps}
        \leq
        \omega_{1}^{\eps}
        \leq
        C
        \eps^{-1},
        \qquad
        \frac{1}{\sinh(\omega_{1}^{\eps})}
        \leq
        C e^{-\omega_{1}^{\eps}}.
    \end{equation*}
    Hence Lemma~\ref{lem:d_infinite} implies
    \begin{equation}\label{eq:d1eps}
        d_{1}^{\eps}
        =
        \frac{\lambda_{1}^{\eps} \omega_{1}^{\eps}}{\eps \pi^2 \sinh(\omega_1^\eps)}
        \leq
        C \eps^{-3}
        \exp
        \left(
            -\frac{1}{\sqrt{2} \eps}
        \right).
    \end{equation}
    Combining~\eqref{eq:first_product_tail} and~\eqref{eq:d1eps} proves~\eqref{eq:first_product_bound}.
\end{proof}

We now prove the first implication~\eqref{thm:main_negative_blowup} in Theorem~\ref{thm:main_negative}. By Proposition~\ref{prop:duality}, it suffices to show the following.

\begin{thm}\label{thm:counterexample_observation_negative}
    For any $0<T<2+2\sqrt{2}$, there exists a positive constant $C>0$ and a family of solutions $(z^{\eps})_{\eps>0}$ to~\eqref{eq:transport-diffusion-adjoint} with $z^{\eps}(0,\cdot) \in H_{0}^{1}(0,1) \cap H^2(0,1)$ such that
    \begin{equation}
        \frac
        {\| z^{\eps}(T,\cdot) \|_{L^2(0,1)}^{2}}
        {\| \eps \partial_x z^{\eps}(t,0) \|_{L^2(0,T)}^{2}}
        \geq
        C \eps^9
        \exp
        \left(
            \frac{2 + 2 \sqrt{2} - T}{2 \eps}
        \right)
    \end{equation}
    for sufficiently small $\eps>0$. In particular, we have the bound
    \begin{equation}\label{eq:cost_blowup_negative}
        C_{\eps}^{-1}(T)
        =
        C_{\eps}(T,1,-1)
        \geq
        \sqrt{C}
        \eps^{9/2}
        \exp
        \left(
            \frac{2 + 2 \sqrt{2} - T}{4 \eps}
        \right).
    \end{equation}
\end{thm}

\begin{proof}
    For $0 < \eps \leq 1 / 2$, consider the vector
    \begin{equation}\label{def:c}
        \bm{c}^{\eps}
        \coloneqq
        G^{-1} \bm{e}_1^{\eps} \in \R^{N_{\eps}},
    \end{equation}
    where $(\bm{e}_{n}^{\eps})_{n=1}^{N_{\eps}}$ is the standard basis of $\R^{N_{\eps}}$. By Proposition~\ref{prop:Gram}, the $n$-th component of $\bm{c}^{\eps}$ is given by
    \begin{equation}\label{components_c}
        \bm{c}_n^{\eps}
        =
        \left(
            G^{-1}
        \right)_{n1}
        =
        \frac{
        4(-1)^{n+1}
        \lambda_{n}^{\eps} \lambda_{1}^{\eps}
        }
        {b_n^{\eps} b_1^{\eps} (\lambda_{n}^{\eps} + \lambda_{1}^{\eps}) d_{n,N_{\eps}}^{\eps} d_{1,N_{\eps}}^{\eps}
        }
        \qquad
        (1 \leq n \leq N_{\eps}).
    \end{equation}
    In particular,
    \begin{equation}\label{eq:c_sgn}
        \operatorname{sgn}\bm{c}_n^{\eps}
        =
        (-1)^{n+1}.
    \end{equation}
    Now define the family of solutions $(z^{\eps})_{\eps>0}$ to~\eqref{eq:transport-diffusion-adjoint} by
    \begin{equation}\label{eq:counterexample_data_negative}
        z^{\eps}(t,x)
        \coloneqq
        \sum_{n=1}^{N_{\eps}}
        \bm{c}_n^{\eps} e^{-\lambda_{n}^{\eps} t} \varphi_{n}^{\eps}(x)
        =
        e^{x/(2 \eps)}
        \sum_{n=1}^{N_{\eps}}
        \bm{c}_n^{\eps} e^{-\lambda_{n}^{\eps} t} \sin(n \pi x).
    \end{equation}
    Obviously, $z^{\eps}(0,\cdot) \in H_{0}^{1}(0,1) \cap H^2(0,1)$. By~\eqref{eq:negative_flux_gram_norm} and~\eqref{def:c}, we have
    \begin{equation}\label{eq:flux_upper_negative}
        \int_{0}^{\infty}
        | \eps \partial_x z^{\eps}(t,0) |^2
        \, dt
        =
        (\bm{c}^{\eps})^T G \bm{c}^{\eps}
        =
        \bm{c}_{1}^{\eps}.
    \end{equation}
    On the other hand, for $1 - \eps^2 < x < 1$ and $1 \leq n \leq N_{\eps} = \lfloor 1 / \eps^2 \rfloor$, we have from~\eqref{eq:c_sgn} that
    \begin{equation}
        \bm{c}_{n}^{\eps} \sin(n \pi x)
        =
        (-1)^{n+1}
        \bm{c}_{n}^{\eps}
        \sin(n \pi (1-x))
        >0.
    \end{equation}
    Hence all the terms on the right-hand side of~\eqref{eq:counterexample_data_negative} have the same sign. In particular, there exists a constant $C>0$ such that
    \begin{align}\label{eq:interior_norm_lower_negative}
        \begin{aligned}
            \int_{0}^{1}
            | z^{\eps}(T,x) |^2
            \, dx
            &
            \geq
            \int_{1-\eps^2}^{1}
            | z^{\eps}(T,x) |^2
            \, dx
            \\
            &
            \geq
            \int_{1-\eps^2}^{1}
            e^{x/\eps} (\bm{c}_{1}^{\eps})^2 e^{-2 \lambda_{1}^{\eps} T} \sin^2(\pi x)
            \, dx
            \\
            &
            \geq
            C \eps^6 (\bm{c}_{1}^{\eps})^2
            \exp
            \left(
                \frac{2-T}{2 \eps}
            \right).
        \end{aligned}
    \end{align}
    From~\eqref{eq:flux_upper_negative} and~\eqref{eq:interior_norm_lower_negative} together with~\eqref{components_c}, it follows that
    \begin{align}
        \frac
        {\| z^{\eps}(T,\cdot) \|_{L^2(0,1)}^{2}}
        {\| \eps \partial_x z^{\eps}(t,0) \|_{L^2(0,T)}^{2}}
        &
        \geq
        C \eps^6 |\bm{c}_{1}^{\eps}|
        \exp
        \left(
            \frac{2-T}{2 \eps}
        \right)
        \\
        &
        =
        \frac
        {2C \eps^6 \lambda_{1}^{\eps}}
        {(b_{1}^{\eps} d_{1,N_{\eps}}^{\eps})^2}
        \exp
        \left(
            \frac{2-T}{2 \eps}
        \right)
        \\
        &
        \geq
        \frac
        {C' \eps^3}
        {(d_{1,N_{\eps}}^{\eps})^2}
        \exp
        \left(
            \frac{2-T}{2 \eps}
        \right)
    \end{align}
    for some $C'>0$. Then applying Lemma~\ref{lem:first_product_bound}, we finally obtain
    \begin{equation}
        \frac
        {\| z^{\eps}(T,\cdot) \|_{L^2(0,1)}^{2}}
        {\| \eps \partial_x z^{\eps}(t,0) \|_{L^2(0,T)}^{2}}
        \geq
        C'' \eps^9
        \exp
        \left(
            \frac{2 + 2\sqrt{2} - T}{2 \eps}
        \right)
    \end{equation}
    for some $C''>0$. The blow-up bound~\eqref{eq:cost_blowup_negative} follows from this and Proposition~\ref{prop:duality}. This ends the proof.
\end{proof}

\section{Infinite-time observability and the boundary-to-interior map}\label{sec:reduction_kernel}
Having established the lower bounds for both signs, we now turn to the proof of the remaining half of Theorems~\ref{thm:main_positive} and~\ref{thm:main_negative}. In this section, we prove some essential tools used in the proof.

\subsection{Reduction to infinite-time observability}\label{sec:reduction}
We first establish that for functions in the space
\begin{equation}
    \mathcal{E}_{\eps,N}
    \coloneqq
    \Span
    \left\{
        e^{-\lambda_{n}^{\eps} t} : 1 \leq n \leq N
    \right\}
    \qquad
    (N \geq 1),
\end{equation}
the infinite-time observation can be bounded by the observation over $[0,T]$, uniformly in $\eps$ and $N$, whenever $T>2$.

\begin{lemma}\label{prop:finite-infinite}
    For every $T>2$, there exists $C_T>0$ such that, for all $\eps>0$, $N \geq 1$, and $f \in \mathcal{E}_{\eps,N}$, we have
    \begin{equation}\label{eq:finite-infinite}
        \int_{0}^{\infty} |f(t)|^2 \, dt \leq C_T \int_{0}^{T} |f(t)|^2 \, dt.
    \end{equation}
\end{lemma}

\begin{proof}
    Note that by Proposition~\ref{prop:model_Theta}, the space $\mathcal{L}(\mathcal{E}_{\eps,N})$ coincides with the model space $K_{\Theta_{\eps,N}}$ of the finite Blaschke product
    \begin{equation}
        \Theta_{\eps,N}
        \coloneqq
        \prod_{n=1}^{N}
        \frac{\lambda_{n}^{\eps}-s}{\lambda_{n}^{\eps}+s}.
    \end{equation}
    By the Bernstein-type inequality for model spaces in~\cite[Theorem~5.2]{BaranovJamingKellaySpeckbacher2024}, it holds that
    \begin{equation}\label{eq:Bernstein}
        \| F' \|_{H^2} \leq \| \Theta_{\eps,N}'(i \cdot) \|_{L^\infty(\mathbb{R})} \| F \|_{H^2}
    \end{equation}
    for $F \in K_{\Theta_{\eps,N}}$.\footnote{We note that D'yakonov~\cite{Dyakonov1994} proved a similar inequality earlier with a non-explicit constant.}~Note that an equation similar to~\eqref{eq:theta_boundary_arctan} holds for $\Theta_{\eps,N}$ as well, and we have
    \begin{equation}
        |\Theta_{\eps,N}'(iy)|
        =
        2
        \sum_{n=1}^{N}
        \frac{\lambda_{n}^{\eps}}{(\lambda_{n}^{\eps})^2 + y^2} \qquad (y \in \R).
    \end{equation}
    Hence
    \begin{equation}
        \| \Theta_{\eps,N}'(i \cdot) \|_{L^\infty(\mathbb{R})}
        =
        | \Theta_{\eps,N}'(0) |
        =
        \sum_{n=1}^{N}\frac{2}{\lambda_{n}^{\eps}}.
    \end{equation}
    The same integral comparison as in~\eqref{eq:deps_lower_upper} yields
    \begin{equation}\label{eq:theta_der_bound}
        \| \Theta_{\eps,N}'(i \cdot) \|_{L^\infty(\mathbb{R})}
        \leq
        2.
    \end{equation}
    If we write $f=\mathcal{L}^{-1}[F]$, we have $tf = -\mathcal{L}^{-1}[F']$. Therefore,~\eqref{eq:Bernstein} and~\eqref{eq:theta_der_bound} together with the isometric property in Proposition~\ref{prop:Laplace_isometry} yield
    \begin{equation}
        \int_{0}^{\infty} t^2 |f(t)|^2 \, dt
        \leq
        4
        \int_{0}^{\infty} |f(t)|^2 \, dt.
    \end{equation}
    This implies that
    \begin{equation}
        T^2 \int_{T}^{\infty} |f(t)|^2 \, dt
        \leq
        \int_{T}^{\infty} t^2 |f(t)|^2 \, dt
        \leq
        \int_{0}^{\infty} t^2 |f(t)|^2 \, dt
        \leq
        4
        \int_{0}^{\infty} |f(t)|^2 \, dt.
    \end{equation}
    Noting that $\mathcal{E}_{\eps,N} = \mathcal{L}^{-1}(K_{\Theta_{\eps,N}})$, this proves that~\eqref{eq:finite-infinite} holds with 
    \begin{equation}
        C_T
        =
        \frac{T^2}{T^2 - 4},
    \end{equation}
    which is positive since $T>2$.
\end{proof}

\subsection{Boundary-to-interior map}\label{sec:kernel}
We next study the relation between the boundary flux and the interior state. For a natural number $N$ and a sequence of complex numbers $(c_n)_{n=1}^{N}$, let
\begin{equation}
    q(t)
    =
    \sum_{n=1}^{N}
    b_{n}^{\eps} c_n e^{-\lambda_{n}^{\eps} t},
    \qquad
    z(t,x)
    =
    \sum_{n=1}^{N}
    c_n e^{-\lambda_{n}^{\eps} t} \varphi_{n}^{\eps,\eta}(x),
\end{equation}
where $\varphi_{n}^{\eps,\eta}$ and $b_{n}^{\eps}$ are defined by~\eqref{eq:phi_signed} and~\eqref{eq:boundary_coefficients}. Note that $q$ is the scaled boundary flux of $z$ at $x=0$:
\begin{equation}
    q(t) = \eps \partial_x z(t,0).
\end{equation}
For $T>0$ and $\eta = \pm 1$, we define the map
\begin{equation}
    \mathcal{R}_{\eps,N}^{\eta}(T)
    \colon
    \mathcal{E}_{\eps,N}
    \to
    L^2(0,1)
\end{equation}
by the formula
\begin{equation}
    \mathcal{R}_{\eps,N}^{\eta}(T) q
    \coloneqq
    z(T,\cdot)
    =
    \sum_{n=1}^{N}
    c_n e^{-\lambda_{n}^{\eps} T} \varphi_{n}^{\eps,\eta}.
\end{equation}
Since $b_{n}^{\eps} \neq 0$ and $(e^{-\lambda_{n}^{\eps} t})_{n=1}^{N}$ are linearly independent in $L^2(0,\infty)$, the map is well-defined. We call $\mathcal{R}_{\eps,N}^{\eta}(T)$ the boundary-to-interior map.

By definition of the operator norm, we have
\begin{equation}\label{eq:BtoI_op}
    \int_{0}^{1} |z(T,x)|^2 \, dx
    \leq
    \| \mathcal{R}_{\eps,N}^{\eta}(T) \|_{\mathrm{op}}^2
    \int_{0}^{\infty} |q(t)|^2 \, dt.
\end{equation}
Thus $\| \mathcal{R}_{\eps,N}^{\eta}(T) \|_{\mathrm{op}}$ gives the infinite-time observability constant. To obtain an infinite-time observability estimate, we use
\begin{equation}\label{eq:BtoI_op_HS}
    \| \mathcal{R}_{\eps,N}^{\eta}(T) \|_{\mathrm{op}}
    \leq
    \| \mathcal{R}_{\eps,N}^{\eta}(T) \|_{\mathrm{HS}},
\end{equation}
where $\| \mathcal{R}_{\eps,N}^{\eta}(T) \|_{\mathrm{HS}}$ is the Hilbert--Schmidt norm of $\mathcal{R}_{\eps,N}^{\eta}(T)$. To evaluate the Hilbert--Schmidt norm, we first derive a kernel representation of $\mathcal{R}_{\eps,N}^{\eta}(T)$ in the model space $K_{\Theta_{\eps,N}} = \mathcal{L}(\mathcal{E}_{\eps,N})$. More precisely, for the operator
\begin{equation}
    \widetilde{\mathcal R}_{\eps,N}^{\eta}(T)
    \coloneqq
    \mathcal{R}_{\eps,N}^{\eta}(T) \circ \mathcal{L}^{-1}
    \colon
    K_{\Theta_{\eps,N}}
    \to
    L^2(0,1),
\end{equation}
we have the following proposition.

\begin{prop}\label{prop:boundary_to_interior_kernel}
    For $(s,x) \in \C_+ \times(0,1)$, define
    \begin{equation}\label{eq:boundary_to_interior_kernel}
        h_{\eps,N,T}^{\eta}(s,x)
        \coloneqq
        \sum_{n=1}^{N}
        \frac{e^{-\lambda_n^{\eps} T} \varphi_n^{\eps,\eta}(x)}{b_n^{\eps}}
        \frac{\Theta_{\eps,N}(s)}
        {\Theta_{\eps,N}'(\lambda_n^{\eps})(s-\lambda_n^{\eps})},
    \end{equation}
    where the apparent singularities at $s=\lambda_n^{\eps}$ are removable since $\Theta_{\eps,N}(\lambda_{n}^{\eps})=0$. Then $h_{\eps,N,T}^{\eta}(\cdot,x) \in K_{\Theta_{\eps,N}}$ and
    \begin{equation}\label{eq:boundary_to_interior_kernel_representation}
        \left(
            \widetilde{\mathcal R}_{\eps,N}^{\eta}(T)F
        \right)
        (x)
        =
        \left\langle
            F,h_{\eps,N,T}^{\eta}(\cdot,x)
        \right\rangle_{H^2}
    \end{equation}
    holds for any $F\in K_{\Theta_{\eps,N}}$. Moreover,
    \begin{equation}\label{eq:boundary_to_interior_norms}
        \| \widetilde{\mathcal R}_{\eps,N}^{\eta}(T) \|_{\mathrm{op}}
        =
        \| \mathcal R_{\eps,N}^{\eta}(T) \|_{\mathrm{op}},
        \qquad
        \| \widetilde{\mathcal R}_{\eps,N}^{\eta}(T) \|_{\mathrm{HS}}
        =
        \|\mathcal R_{\eps,N}^{\eta}(T)\|_{\mathrm{HS}}.
    \end{equation}
\end{prop}

\begin{proof}
    Fix $x \in (0,1)$ and abbreviate $\Theta = \Theta_{\eps,N}$ and $h_x = h_{\eps,N,T}^{\eta}(\cdot,x)$. Since $\Theta(s) / (s-\lambda_n^{\eps})$ is a proper rational function with simple poles at $-\lambda_1^{\eps},\ldots,-\lambda_N^{\eps}$, partial fractions imply that it belongs to $K_{\Theta}$. Hence $h_x \in K_{\Theta}$. To prove the kernel representation~\eqref{eq:boundary_to_interior_kernel_representation}, we use the canonical conjugation on $K_{\Theta}$, which is the conjugate-linear map
    \begin{equation}
        C
        \colon
        K_{\Theta}
        \to
        K_{\Theta}
    \end{equation}
    defined for $F(s) = \sum_{n=1}^{N} a_n / (s+\lambda_n^{\eps}) \in K_{\Theta}$ by
    \begin{equation}
        (CF)(s)
        =
        \Theta(s)
        \sum_{n=1}^{N} \frac{\overline{a_n}}{\lambda_{n}^{\eps} - s}.
    \end{equation}
    The same partial-fraction argument shows that $CF \in K_\Theta$. On the boundary $i \R = \partial \C_+$, we have
    \begin{equation}
        (CF)(iy)
        =
        \Theta(iy)
        \widebar{F(iy)}.
    \end{equation}
    Since $|\Theta(iy)|=1$, we have
    \begin{equation}
        (C^2 F)(iy) = F(iy), \qquad \| CF \|_{H^2} = \| F \|_{H^2}.
    \end{equation}
    Note that by Proposition~\ref{prop:boundary_trace}, a function in $H^2$ is uniquely defined by the boundary values. Hence the map $C$ is an isometric involution, that is, isometric with $C^2 F = F$. In particular,
    \begin{equation}\label{eq:involutive_identity}
        \left\langle
            F,h_x
        \right\rangle_{H^2}
        =
        \left\langle
            Ch_x,CF
        \right\rangle_{H^2}.
    \end{equation}

    Now let
    \begin{equation}
        F(s)
        =
        \sum_{n=1}^{N}
        \frac{b_n^{\eps}c_n}{s+\lambda_n^{\eps}}
        \qquad
        (c_n \in \C)
    \end{equation}
    and
    \begin{equation}
        q(t)
        =
        \mathcal{L}^{-1}[F](t)
        =
        \sum_{n=1}^{N} b_{n}^{\eps} c_n e^{-\lambda_{n}^{\eps} t}.
    \end{equation}
    Evaluation of $CF$ at the simple zeros $s = \lambda_{n}^{\eps}$ of $\Theta$ yields
    \begin{equation}\label{eq:conjugation_coefficient_recovery}
        (CF)(\lambda_n^{\eps})
        =
        -
        \Theta'(\lambda_n^{\eps})
        b_n^{\eps}
        \overline{c_n}.
    \end{equation}
    On the other hand, by~\eqref{eq:boundary_to_interior_kernel}, we have
    \begin{equation}\label{eq:Ch_x}
        (Ch_x)(iy)
        =
        -
        \sum_{n=1}^{N}
        \frac{
            e^{-\lambda_n^{\eps} T}
            \varphi_n^{\eps,\eta}(x)
        }
        {
            b_n^{\eps} \Theta'(\lambda_n^{\eps})
        }
        \frac{1}{\lambda_n^{\eps} + iy}.
    \end{equation}
    Hence Proposition~\ref{prop:reproducing_kernel} and~\eqref{eq:involutive_identity} imply
    \begin{align*}
        \left\langle
            F,h_x
        \right\rangle_{H^2}
        &
        =
        \left\langle
            Ch_x,CF
        \right\rangle_{H^2}
        \\
        &
        =
        -
        \sum_{n=1}^{N}
        \frac{
            e^{-\lambda_n^{\eps} T}
            \varphi_n^{\eps,\eta}(x)
        }
        {
            b_n^{\eps}
            \Theta'(\lambda_n^{\eps})
        }
        \overline{(CF)(\lambda_n^{\eps})}
        \\
        &
        =
        \sum_{n=1}^{N}
        c_n
        e^{-\lambda_n^{\eps} T}
        \varphi_n^{\eps,\eta}(x)
        \\
        &
        =
        \left(
            \mathcal R_{\eps,N}^{\eta}(T) q
        \right)(x)
        =
        \left(
            \widetilde{\mathcal R}_{\eps,N}^{\eta}(T)F
        \right)(x).
    \end{align*}
    This proves~\eqref{eq:boundary_to_interior_kernel_representation}. Finally, note that $\mathcal L^{-1} \colon K_{\Theta_{\eps,N}} \to \mathcal E_{\eps,N}$ is unitary by Proposition~\ref{prop:Laplace_isometry}. Composition with this map preserves both the operator norm and the Hilbert--Schmidt norm, proving~\eqref{eq:boundary_to_interior_norms}.
\end{proof}

We next compute the Hilbert--Schmidt norm $\| \mathcal{R}_{\eps,N}^{\eta}(T) \|_{\mathrm{HS}}$ using the kernel representation in Proposition~\ref{prop:boundary_to_interior_kernel}. Using the products $d_{n,N}^{\eps}$ defined by~\eqref{eq:finite_product_d_firstdef}, set
\begin{equation}\label{eq:finite_boundary_interior_kernel}
    H_{\eps,N}(t,y)
    \coloneqq
    \sum_{n=1}^{N}
    \frac{\lambda_n^{\eps}}
    {\eps \pi n d_{n,N}^{\eps}}
    e^{-\eps \pi^2 n^2 t}
    \sin(n\pi y)
    \qquad
    (t>0,\ 0<y<1).
\end{equation}

\begin{prop}\label{prop:finite_hs_norm}
    We have
    \begin{equation}\label{eq:finite_hs_norm}
        \|
            \mathcal R_{\eps,N}^{\eta}(T)
        \|_{\mathrm{HS}}^2
        =
        4
        \int_T^{\infty}
        e^{-t/(2\eps)}
        \int_0^1
        e^{-\eta x/\eps}
        \left|
            H_{\eps,N}(t,1-x)
        \right|^2
        \, dx \, dt.
    \end{equation}
\end{prop}

\begin{proof}
    For $x \in (0,1)$, we abbreviate $h_x = h_{\eps,N,T}^{\eta}(\cdot,x)$ and
    $\Theta = \Theta_{\eps,N}$. Let $(e_j)_{j=1}^N$ be an orthonormal
    basis of $K_{\Theta}$. By
    Proposition~\ref{prop:boundary_to_interior_kernel} and Parseval's identity,
    \begin{equation}
        \|
            \mathcal R_{\eps,N}^{\eta}(T)
        \|_{\mathrm{HS}}^2
        =
        \|
            \widetilde{\mathcal R}_{\eps,N}^{\eta}(T)
        \|_{\mathrm{HS}}^2
        =
        \sum_{j=1}^N
        \int_0^1
        \left|
            \left\langle e_j,h_x \right\rangle_{H^2}
        \right|^2
        \, dx
        =
        \int_0^1
        \| h_x \|_{H^2}^2
        \, dx.
    \end{equation}
    By Proposition~\ref{prop:reproducing_kernel} and~\eqref{eq:Ch_x}, the integrand on the right-hand side can be computed as
    \begin{align}
        \| h_x \|_{H^2}^{2}
        =
        \| C h_x \|_{H^2}^{2}
        &
        =
        \sum_{m,n=1}^{N}
        \frac{
            e^{-\left( \lambda_{m}^{\eps} + \lambda_{n}^{\eps} \right) T}
            \varphi_{m}^{\eps,\eta}(x) \varphi_{n}^{\eps,\eta}(x)
        }
        {
            b_{m}^{\eps} b_{n}^{\eps}
            \Theta'(\lambda_{m}^{\eps})
            \Theta'(\lambda_{n}^{\eps})
            (\lambda_{m}^{\eps} + \lambda_{n}^{\eps})
        }
        \\
        &
        =
        \int_{T}^{\infty}
        \left|
        \sum_{n=1}^{N}
        \frac{
            e^{-\lambda_{n}^{\eps} t}
            \varphi_{n}^{\eps,\eta}(x)
        }
        {
            b_{n}^{\eps}
            \Theta'(\lambda_{n}^{\eps})
        }
        \right|^2
        \, dt.
    \end{align}
    Hence
    \begin{equation}\label{eq:hs_norm_spectral_kernel}
        \|
            \mathcal R_{\eps,N}^{\eta}(T)
        \|_{\mathrm{HS}}^2
        =
        \int_T^{\infty}
        \int_0^1
        \left|
        \sum_{n=1}^{N}
        \frac{
            e^{-\lambda_{n}^{\eps} t}
            \varphi_{n}^{\eps,\eta}(x)
        }
        {
            b_{n}^{\eps}
            \Theta'(\lambda_{n}^{\eps})
        }
        \right|^2
        \, dx \, dt.
    \end{equation}
    Note that
    \begin{equation}\label{eq:finite_blaschke_derivative_d}
        \Theta'(\lambda_n^{\eps})
        =
        \frac{(-1)^n}{2\lambda_n^{\eps}}
        d_{n,N}^{\eps}
    \end{equation}
    and
    \begin{equation*}
        \sin(n\pi(1-x))=(-1)^{n+1}\sin(n\pi x).
    \end{equation*}
    Then plugging in the formulas for $\lambda_n^{\eps}$,
    $\varphi_n^{\eps,\eta}$, and $b_n^{\eps}$ gives
    \begin{equation}
        \sum_{n=1}^N
        \frac{
            e^{-\lambda_n^{\eps} t}
            \varphi_n^{\eps,\eta}(x)
        }
        {
            b_n^{\eps}
            \Theta'(\lambda_n^{\eps})
        }
        =
        -2 e^{-t/(4\eps)} e^{-\eta x/(2\eps)}
        H_{\eps,N}(t,1-x).
    \end{equation}
    Substituting this into~\eqref{eq:hs_norm_spectral_kernel} proves~\eqref{eq:finite_hs_norm}.
\end{proof}

\begin{remark}
    The Hilbert--Schmidt norm formula in Proposition~\ref{prop:finite_hs_norm} can also be derived by matrix calculations. Let $G$ and $B$ be the Gram matrices of
    \begin{equation}
        \left( b_n^\eps e^{-\lambda_n^\eps t} \right)_{n=1}^N
        \subset
        L^2(0,\infty),
        \qquad
        \left(
            e^{-\lambda_n^\eps T} \varphi_n^{\eps,\eta}
        \right)_{n=1}^N
        \subset
        L^2(0,1),
    \end{equation}
    respectively. Namely,
    \begin{equation}
        G_{mn}
        =
        \frac
        {b_m^\eps b_n^\eps}
        {\lambda_m^\eps + \lambda_n^\eps},
        \qquad
        B_{mn}
        =
        e^{-(\lambda_m^\eps + \lambda_n^\eps) T}
        \int_0^1
        \varphi_m^{\eps,\eta}(x)
        \varphi_n^{\eps,\eta}(x)
        \,
        dx.
    \end{equation}
    We note that the Gram matrix $G$ already appeared in Proposition~\ref{prop:Gram}, with $N_{\eps}$ replaced by $N$. For $\bm{c} = (c_n)_{n=1}^{N}$, let $q = \sum_{n=1}^N b_n^\eps c_n e^{-\lambda_n^\eps t}$. Then we have
    \begin{equation}
        \| q \|_{L^2(0,\infty)}^2
        =
        \bm c^* G \bm c,
        \qquad
        \|
            \mathcal R_{\eps,N}^{\eta}(T) q
        \|_{L^2(0,1)}^2
        =
        \bm c^* B \bm c.
    \end{equation}
    Hence the square of the operator norm of the boundary-to-interior map $\mathcal{R}_{\eps,N}^{\eta}(T)$ is given by the maximum of the generalized Rayleigh quotient:
    \begin{equation}
        \|\mathcal
            R_{\eps,N}^{\eta}(T)
        \|_{\mathrm{op}}^2
        =
        \sup_{\bm c \in \C^N \setminus \{ 0 \}}
        \frac{\bm c^* B \bm c}{\bm c^* G \bm c}
        =
        \sup_{\bm c \in \C^N \setminus \{ 0 \}}
        \frac{\bm c^* G^{-1/2} B G^{-1/2} \bm c}{\bm c^* \bm c}.
    \end{equation}
    The right-hand side gives the largest eigenvalue of the matrix $G^{-1/2} B G^{-1/2}$, which can be bounded by its trace. Using the inverse matrix formula in Proposition~\ref{prop:Gram}, after some calculations, we obtain
    \begin{align}
        \operatorname{tr}
        \left(
            G^{-1/2} B G^{-1/2}
        \right)
        &
        =
        \operatorname{tr}
        \left(
            G^{-1} B
        \right)
        \\
        &
        =
        4
        \int_T^{\infty}
        e^{-t/(2\eps)}
        \int_0^1
        e^{-\eta x/\eps}
        \left|
            H_{\eps,N}(t,1-x)
        \right|^2
        \, dx \, dt.
    \end{align}
    The integral representation is identical to that of $\| \mathcal R_{\eps,N}^{\eta}(T) \|_{\mathrm{HS}}^{2}$ in Proposition~\ref{prop:finite_hs_norm}. Indeed, in the basis $(b_n^\eps e^{-\lambda_n^\eps t})_{n=1}^N$, the operator
    $
    \mathcal R_{\eps,N}^{\eta}(T)^*
    \mathcal R_{\eps,N}^{\eta}(T)
    $
    is represented by $G^{-1}B$. Hence
    \begin{equation}
        \|
            \mathcal R_{\eps,N}^{\eta}(T)
        \|_{\mathrm{HS}}^2
        =
        \operatorname{tr}(G^{-1}B).
    \end{equation}
    This gives another proof of Proposition~\ref{prop:finite_hs_norm}.
\end{remark}

We now pass to the limit $N\to\infty$. Recall that
\begin{equation}\label{eq:omega_boundary_interior}
    \omega_n^{\eps}
    =
    \sqrt{\pi^2 n^2 + \frac{1}{2\eps^2}}
\end{equation}
and define
\begin{equation}\label{eq:infinite_boundary_interior_kernel}
    H_{\eps}(t,y)
    \coloneqq
    \sum_{n=1}^{\infty}
    \frac{\pi n}{\omega_n^{\eps}}
    \sinh(\omega_n^{\eps})
    e^{-\eps \pi^2 n^2 t}
    \sin(n\pi y)
    \qquad
    (t>0,\ 0<y<1).
\end{equation}

\begin{prop}\label{prop:limit_hs_norm}
    Set
    \begin{equation}\label{eq:infinite_hs_bound}
        A_{\eps}^{\eta}(T)
        \coloneqq
        4
        \int_{T}^{\infty}
        e^{-t/(2 \eps)}
        \int_{0}^{1}
        e^{-\eta x / \eps}
        \left|
            H_{\eps}(t,1-x)
        \right|^2
        \, dx \, dt.
    \end{equation}
    Then $A_{\eps}^{\eta}(T)<\infty$ and
    \begin{equation}\label{eq:HS_limit}
        A_{\eps}^{\eta}(T)
        =
        \lim_{N \to \infty}
        \| \mathcal{R}_{\eps,N}^{\eta}(T) \|_{\mathrm{HS}}^2
        =
        \sup_{N \geq 1}
        \| \mathcal{R}_{\eps,N}^{\eta}(T) \|_{\mathrm{HS}}^2.
    \end{equation}
\end{prop}

\begin{proof}
    By Lemma~\ref{lem:d_infinite}, we have
    \begin{equation}\label{eq:bound_dn}
        0
        <
        \frac{\lambda_{n}^{\eps}}{\eps \pi n d_{n,N}^{\eps}}
        \leq
        \lim_{N \to \infty}
        \left(
            \frac{\lambda_{n}^{\eps}}{\eps \pi n d_{n,N}^{\eps}}
        \right)
        =
        \frac{\pi n}{\omega_{n}^{\eps}}
        \sinh(\omega_{n}^{\eps}).
    \end{equation}
    Using the inequality $\sqrt{a+b} \leq \sqrt{a} + \sqrt{b}$ for $a, b \geq 0$, we obtain
    \begin{equation}\label{eq:elementary_sqrt_bound}
        \pi n \leq \omega_{n}^{\eps} \leq \pi n + \frac{1}{\sqrt{2} \eps}.
    \end{equation}
    From~\eqref{eq:bound_dn} and~\eqref{eq:elementary_sqrt_bound}, it follows that there exists a constant $C_{\eps}>0$ depending only on $\eps>0$ such that
    \begin{equation}\label{eq:HepsN_Heps_summable_bound}
        |H_{\eps,N}(t,y)| + |H_{\eps}(t,y)|
        \leq
        C_{\eps}
        \sum_{n=1}^{\infty}e^{\pi n - \eps \pi^2 n^2 T}
        \eqqcolon
        B_{\eps,T}
        <
        \infty
    \end{equation}
    for all $t \geq T$, $y \in (0,1)$, and every natural number $N$. Hence the convergence of the coefficients and dominated convergence theorem imply
    \begin{equation}
        \lim_{N \to \infty}
        H_{\eps,N}(t,y)
        =
        H_{\eps}(t,y)
        \qquad
        (t \geq T, \ 0<y<1).
    \end{equation}
    This pointwise convergence and~\eqref{eq:HepsN_Heps_summable_bound} allow us to apply the dominated convergence theorem to the integral representation in Proposition~\ref{prop:finite_hs_norm} to obtain
    \begin{equation}
        \lim_{N \to \infty}
        \|
            \mathcal{R}_{\eps,N}^{\eta}(T)
        \|_{\mathrm{HS}}^{2}
        =
        A_{\eps}^{\eta}(T)
        <
        \infty.
    \end{equation}
    Finally, note that by the inclusion $\mathcal{E}_{\eps,N} \subset \mathcal{E}_{\eps,N+1}$ and
    \begin{equation}
        \mathcal{R}_{\eps,N+1}^{\eta}(T)|_{\mathcal{E}_{\eps,N}}
        =
        \mathcal{R}_{\eps,N}^{\eta}(T),
    \end{equation}
    we have
    \begin{equation}
        \| 
            \mathcal{R}_{\eps,N}^{\eta}(T)
        \|_{\mathrm{HS}}^{2}
        \leq
        \| 
            \mathcal{R}_{\eps,N+1}^{\eta}(T)
        \|_{\mathrm{HS}}^{2}.
    \end{equation}
    This implies the second equality in~\eqref{eq:HS_limit}, completing the proof.
\end{proof}

\subsection{Localization estimates for $H_{\eps}$}\label{sec:localization}
In this section, we estimate the auxiliary function $H_{\eps}$ defined by~\eqref{eq:infinite_boundary_interior_kernel}. As it is expressed as a Fourier series with a smooth and rapidly decaying filter, we expect that it is concentrated near the origin; see for example~\cite[Section~2.2]{FilbirMhaskarPrestin2012}. In fact, we have the following proposition.

\begin{prop}\label{lem:H_localization}
Let $T>2$. Fix $0<a<1/2$ and set $q_a = \sqrt{1/2-a^2}$. There exists a constant $C_{T,a}>0$ such that
\begin{equation}\label{eq:H_localization}
    |H_{\eps}(t,y)|
    \leq
    C_{T,a}
    \eps^{-1/2}
    \exp
    \left(
        \frac{q_a + t a^2 - a y}{\eps}
    \right)
\end{equation}
for $0<\eps \leq 1$, $t \geq T$, and $0<y<1$.
\end{prop}

\begin{proof}
    We first establish a bound for the real part of the square root that will appear later. Set
    \begin{equation}
        u_a(r)
        =
        \Re
        \sqrt{(r-ia)^2 +1/2}
        =
        \Re
        \sqrt{r^2 + q_a^2 - 2 i a r},
        \qquad
        r \in \R.
    \end{equation}
    Note that $u_a(r)>0$. By a simple calculation, we get
    \begin{equation}
        u_{a}(r)^2
        =
        \frac{\sqrt{(r^2 + q_{a}^2)^2 + (-2 a r)^2} + r^2 + q_a^2}{2}.
    \end{equation}
    Writing $B \coloneqq r^2 +q_a^2 > 0$ and using the inequality $\sqrt{1+x} \leq 1 + x / 2$ for $x \geq -1$, we get
    \begin{equation}
        u_{a}(r)^2
        =
        \frac{\sqrt{B^2 + 4 a^2 r^2} + B}{2}
        \leq
        B + \frac{a^2 r^2}{B}
        =
        r^2 + q_{a}^{2} + \frac{a^2 r^2}{r^2 + q_{a}^{2}}.
    \end{equation}
    Using
    \begin{equation}
        \frac{1}{2 q_{a}^{2}}
        -
        1
        -
        \frac{a^2}{r^2 + q_a^2}
        =
        \frac{1-2q_{a}^{2}}{2q_{a}^{2}}
        -
        \frac{a^2}{r^2 + q_a^2}
        =
        \left(
           \frac{1}{q_{a}^{2}}
           -
           \frac{1}{r^2 + q_{a}^{2}}
        \right)
        a^2
        \geq
        0,
    \end{equation}
    we further obtain
    \begin{equation}
        u_{a}(r)^2
        \leq
        q_{a}^{2} + \frac{r^2}{2 q_{a}^{2}}.
    \end{equation}
    Then $q_a>1/2$ and the inequality $\sqrt{1+x} \leq 1 + x / 2$ imply
    \begin{equation}\label{eq:localization_sqrt_bound}
        u_{a}(r)
        \leq
        q_a + \frac{r^2}{4 q_{a}^{3}}
        \leq
        q_a + 2 r^2.
    \end{equation}

    We next rewrite $H_{\eps}$ using the Poisson summation formula. For the Fourier transform, we use the convention
    \begin{equation}
        \widehat{f}(\zeta)
        \coloneqq
        \int_\R
        f(\xi)
        e^{-i \xi \zeta}
        \, d\xi.
    \end{equation}
    With this convention, the Poisson summation formula reads
    \begin{equation}\label{eq:PSF}
        \sum_{n\in\mathbb Z} f(\pi n)
        =
        \frac{1}{\pi}
        \sum_{j\in\mathbb Z}\widehat{f}(2j),
        \qquad
        f \in \mathcal S(\mathbb R);
    \end{equation}
    see~\cite[Vol.~I, Chapter~II, \S13]{Zygmund2002}. Let
    \begin{equation}
        b_{\eps,t}(\xi)
        \coloneqq
        \frac{
            \sinh \sqrt{\xi^2+\frac{1}{2\eps^2}}
        }
        {
            \sqrt{\xi^2+\frac{1}{2\eps^2}}
        }
        e^{-\eps t\left( \xi^2+\frac{1}{2\eps^2} \right)}.
    \end{equation}
    For any $\eps, t>0$, this is an even Schwartz function. Applying the Poisson summation formula~\eqref{eq:PSF} with $f(\xi) = b_{\eps,t}(\xi) e^{i \xi y}$, we obtain
    \begin{equation}\label{eq:PSF_applied}
        \sum_{n\in\mathbb{Z}}
        b_{\eps,t}(\pi n)e^{i \pi n y}
        =
        \frac{1}{\pi}
        \sum_{j\in\mathbb{Z}}
        \widehat b_{\eps,t}(2j-y)
        =
        \frac{1}{\pi}
        \sum_{j\in\mathbb{Z}}
        \widehat b_{\eps,t}(y+2j).
    \end{equation}
    Differentiation with respect to $y$ yields
    \begin{equation}\label{eq:PSF_applied_diff}
        \sum_{n\in\mathbb{Z}}
        i \pi n
        b_{\eps,t}(\pi n)e^{i \pi n y}
        =
        \frac{1}{\pi}
        \sum_{j\in\mathbb{Z}}
        \partial_{\zeta} \widehat b_{\eps,t}(y+2j).
    \end{equation}
    Hence the expression of $H_{\eps}$ in~\eqref{eq:infinite_boundary_interior_kernel} can be rewritten as
    \begin{equation}\label{eq:kernel_poisson}
        H_{\eps}(t,y)
        =
        -\frac{e^{t/(2\eps)}}{2}
        \sum_{n \in \Z}
        i \pi n b_{\eps,t}(\pi n) e^{i \pi n y}
        =
        -\frac{e^{t/(2\eps)}}{2 \pi}
        \sum_{j \in \Z}
        \partial_{\zeta} \widehat{b}_{\eps,t}(y+2j).
    \end{equation}

    We now prove a decay estimate of $\partial_{\zeta} \widehat b_{\eps,t}(\zeta)$. Writing $\rho(s) \coloneqq \sqrt{s^2+1/2}$ and making the change of variables as $\xi = s / \eps$, we get
    \begin{equation}\label{eq:kernel_fourier_integral}
        \widehat b_{\eps,t}(\zeta)
        =
        \int_\R
        \frac{\sinh(\rho(s)/\eps)}{\rho(s)}
        e^{-t(s^2 + 1/2) / \eps}
        e^{-i s \zeta / \eps}
        \, ds.
    \end{equation}
    Since the term in front of the oscillatory factor $e^{-i s \zeta / \eps}$ is entire in $s$ and decays rapidly as $|\Re s| \to \infty$ in any fixed horizontal strip, we are allowed to shift the contour of integration vertically. As in the method of steepest descent, this converts the oscillatory factor into exponential decay. First, consider $\zeta \geq 0$. Writing $s = r - i a$ with $r$ and $a$ real, we have
    \begin{equation}
        \left|
            \frac{s}{\rho(s)}
        \right|^2
        =
        \frac{r^2+a^2}{|(r-ia)^2+1/2|}
        \leq
        \frac{r^2+a^2}{r^2+q_a^2}
        \leq
        1.
    \end{equation}
    Moreover, since $|\sinh(z)| \leq e^{\Re z}$ for $\Re z \geq 0$,
    \begin{equation}
        \left|
            \sinh(\rho(s)/\eps)
        \right|
        \leq
        e^{u_a(r) / \eps}.
    \end{equation}
    Consequently, deforming the contour of integration to $\R - i a$ and using~\eqref{eq:localization_sqrt_bound}, we obtain
    \begin{align}\label{eq:kernel_fourier_bound}
        \begin{aligned}
            e^{t/(2\eps)}
            \left|
                \partial_\zeta \widehat b_{\eps,t}(\zeta)
            \right|
            &
            \leq
            \frac{1}{\eps}
            e^{(q_a + t a^2 - a \zeta) / \eps}
            \int_\R
            e^{-(t-2) r^2 / \eps}
            \, dr
            \\
            &
            =
            \frac{\sqrt{\pi}}{\sqrt{\eps (t-2)}}
            e^{(q_a + t a^2 - a \zeta) / \eps}.
        \end{aligned}
    \end{align}
    By evenness of $b_{\eps,t}$, the same estimate holds for all $\zeta \in \R$ with $\zeta$ replaced by $|\zeta|$. Combining this with~\eqref{eq:kernel_poisson},~\eqref{eq:kernel_fourier_bound}, and the bound
    \begin{align}
        \sum_{j \in \Z}
        e^{-a | y + 2j | / \eps}
        &
        =
        \sum_{j=0}^{\infty}
        e^{-a (y + 2j) / \eps}
        +
        \sum_{j=1}^{\infty}
        e^{-a (2j - y) / \eps}
        \\
        &
        =
        \frac{e^{- a y / \eps} + e^{- a (2 - y) / \eps}}{1 - e^{-2a / \eps}}
        \\
        &
        \leq
        \frac{2 e^{-a y / \eps}}{1 - e^{-2a}}
        \qquad
        (0<y<1),
    \end{align}
    we obtain~\eqref{eq:H_localization} with the choice
    \begin{equation}
        C_{T,a}
        =
        \frac{1}{\sqrt{\pi (T-2)} (1-e^{-2a})}.
    \end{equation}
\end{proof}

\section{Uniform observability via boundary-to-interior estimates}\label{sec:proof_upper}
With the preparation in the previous section, we are now ready to complete the proof of Theorems~\ref{thm:main_positive} and~\ref{thm:main_negative}. We begin with the positive-speed case.

\begin{thm}\label{prop:positive_cost_bound}
    For every $T>2$, there exists a constant $K_T>0$ such that
    \begin{equation}\label{eq:positive_cost_bound}
        C_{\eps}^{1}(T)
        =
        C_{\eps}(T,1,1)
        \leq
        K_T
        \sqrt{\eps}
        \exp
        \left(
            -\frac{(T-2)^2}{4T \eps}
        \right)
        \qquad
        (0 < \eps \leq 1),
    \end{equation}
    where the cost $C_{\eps}(T,1,1)$ is defined by~\eqref{eq:cost}.
\end{thm}

By Proposition~\ref{prop:duality} and the scaling identity~\eqref{eq:scaling_cost}, Theorem~\ref{prop:positive_cost_bound} proves~\eqref{thm:main_positive_zero}. Together with the lower bound proved in Section~\ref{sec:proof}, it therefore completes the proof of Theorem~\ref{thm:main_positive}.

\begin{proof}
    Fix $T>2$. We first establish the bound
    \begin{equation}\label{eq:cost_bound_by_A}
        C_{\eps}^{\eta}(T)^2
        \leq
        C_T
        A_{\eps}^{\eta}(T)
        \qquad
        (\eps>0, \ \eta = \pm 1),
    \end{equation}
    where $C_T$ is the constant in~\eqref{eq:finite-infinite}. Let $z_N$ be the solution to~\eqref{eq:transport-diffusion-adjoint} with the initial data
    \begin{equation}
        z_{0,N}
        =
        \sum_{n=1}^{N}
        c_n \varphi_{n}^{\eps,\eta}.
    \end{equation}
    Write $q_N(t) = \eps \partial_x z_N(t,0)$. Then by~\eqref{eq:BtoI_op},~\eqref{eq:BtoI_op_HS}, Lemma~\ref{prop:finite-infinite}, and Proposition~\ref{prop:limit_hs_norm}, we obtain
    \begin{align}\label{eq:finite_sum_observability_A}
        \begin{aligned}
            \int_{0}^{1}
            |z_N(T,x)|^2
            \, dx
            &
            \leq
            \|
                \mathcal R_{\eps,N}^{\eta}(T)
            \|_{\mathrm{op}}^2
            \int_{0}^{\infty}
            |q_N(t)|^2
            \, dt
            \\
            &
            \leq
            A_{\eps}^{\eta}(T)
            \int_{0}^{\infty}
            |q_N(t)|^2
            \, dt
            \\
            &
            \leq
            C_T A_{\eps}^{\eta}(T)
            \int_{0}^{T}
            |q_N(t)|^2
            \, dt
            \\
            &
            =
            C_T A_{\eps}^{\eta}(T)
            \int_{0}^{T} | \eps \partial_x z_N(t,0) |^2 \, dt.
        \end{aligned}
    \end{align}
    This estimate extends to every initial data $z_0 \in H_0^1(0,1) \cap H^2(0,1)$. In fact, expand $e^{\eta x / (2 \eps)} z_0$ in sine series as
    \begin{equation}
        e^{\eta x / (2 \eps)} z_0
        =
        \sum_{n=1}^{\infty}
        c_n
        \sin(n \pi x).
    \end{equation}
    Since $e^{-\eta x / (2 \eps)} \in W^{2,\infty}(0,1)$, the partial sum
    \begin{equation}
        z_{0,N}
        \coloneqq
        \sum_{n=1}^{N}
        c_n
        e^{-\eta x / (2 \eps)}
        \sin(n \pi x)
        =
        \sum_{n=1}^{N}
        c_n \varphi_{n}^{\eps,\eta}
    \end{equation}
    converges as $N \to \infty$ to $z_0$ in $H^2(0,1)$. Let $z$ and $z_N$ be the solutions to~\eqref{eq:transport-diffusion-adjoint} with the initial data $z_0$ and $z_{0,N}$, respectively. Then
    \begin{equation}
        z_N
        =
        \sum_{n=1}^{N}
        c_n e^{-\lambda_{n}^{\eps} t}
        \varphi_{n}^{\eps,\eta}
        \to
        z
        =
        \sum_{n=1}^{\infty}
        c_n e^{-\lambda_{n}^{\eps} t}
        \varphi_{n}^{\eps,\eta}       
    \end{equation}
    as $N \to \infty$ in $C([0,T];H^2(0,1))$. Continuity of the trace map $v \mapsto v_x(0)$ on $H^2(0,1)$ also implies
    \begin{equation}
        \eps \partial_x z_N(\cdot,0)
        \to
        \eps \partial_x z(\cdot,0)
    \end{equation}
    as $N \to \infty$ in $L^2(0,T)$. Passing to the limit in~\eqref{eq:finite_sum_observability_A}, we get
    \begin{equation}
        \int_{0}^{1}
        |z(T,x)|^2
        \, dx
        \leq
        C_T A_{\eps}^{\eta}(T)
        \int_{0}^{T}
        |\eps \partial_x z(t,0)|^2
        \, dt.
    \end{equation}
    This combined with Proposition~\ref{prop:duality} proves~\eqref{eq:cost_bound_by_A}.
    
    We now take $\eta = 1$ and estimate $A_{\eps}^{1}(T)$ using Proposition~\ref{lem:H_localization}. Set $a = 1 /T$. Since $T>2$, we have $0<a<1/2$. In particular, $q_a = \sqrt{1/2-a^2} < 1 - a$. Hence by Proposition~\ref{lem:H_localization}, we have
    \begin{equation}
        \left|
            H_{\eps}(t,1-x)
        \right|^2
        \leq
        C_{T,a}^2
        \eps^{-1}
        \exp
        \left(
            \frac{2 - 4 a + 2 t a^2 + 2 a x}{\eps}
        \right)
    \end{equation}
    for $t \geq T$ and $0 < x < 1$. Inserting this into~\eqref{eq:infinite_hs_bound} yields
    \begin{align}
        A_{\eps}^{1}(T)
        &
        \leq
        4 C_{T,a}^2 \eps^{-1}
        e^{(2 - 4 a) / \eps}
        \int_T^{\infty}
        e^{-\alpha t / \eps}
        \, dt
        \int_0^1
        e^{-\beta x / \eps}
        \, dx
        \\
        &
        \leq
        \frac{4 C_{T,a}^2}{\alpha \beta}
        \eps
        \exp
        \left(
            \frac{2 - 4 a - \alpha T}{\eps}
        \right),
    \end{align}
    where
    \begin{equation}
        \alpha
        \coloneqq
        \frac{1}{2}
        -
        2 a^2>0,
        \qquad
        \beta
        \coloneqq
        1
        -
        2a>0.
    \end{equation}
    The exponent
    \begin{equation}
        2 - 4 a - \alpha T
        =
        2 - \frac{T}{2} - 4 a + 2 T a^2
    \end{equation}
    is minimized at $a = 1 / T \in (0,1/2)$. For this choice $a = 1 / T$, we obtain
    \begin{equation}
        2 - 4 a - \alpha T
        =
        2 - \frac{T}{2} - \frac{2}{T}
        =
        -\frac{(T-2)^2}{2T}.
    \end{equation}
    Combining this estimate with~\eqref{eq:cost_bound_by_A} proves~\eqref{eq:positive_cost_bound} with
    \begin{equation}
        K_T
        =
        2 C_{T,a} \sqrt{\frac{C_T}{\alpha \beta}}.
    \end{equation}
\end{proof}

We next prove the upper bound for negative speed. By the scaling identity~\eqref{eq:scaling_cost}, it suffices to prove the following.

\begin{thm}\label{prop:negative_cost_bound}
    For every $T \geq 2+2\sqrt{2}$, there exists a constant $K_T>0$ such that
    \begin{equation}\label{eq:negative_cost_bound}
        C_{\eps}^{-1}(T)
        =
        C_{\eps}(T,1,-1)
        \leq
        K_T
        \sqrt{\eps}
        \exp
        \left(
            -\frac{T - 2 - 2 \sqrt{2}}{4 \eps}
        \right)
        \qquad
        (0 < \eps \leq 1).
    \end{equation}
\end{thm}

\begin{remark}
    As in the positive-speed case, we may apply Proposition~\ref{lem:H_localization} to obtain a similar bound. However, a cruder and more direct estimate suffices for the negative-speed case, as we shall see below. We also remark that an application of~\cite[Theorem~2.1]{Jacob2006} proves $\lim_{\eps \to 0}C_{\eps}^{-1}(T)=0$ for $T>2+2\sqrt{2}$.
\end{remark}

\begin{proof}[Proof of Theorem~\ref{prop:negative_cost_bound}]
    From the definition of $\omega_{n}^{\eps}$ in~\eqref{eq:omega_boundary_interior}, we have
    \begin{equation}
        \pi n \leq \omega_{n}^{\eps} \leq \frac{1}{\sqrt{2} \eps} + \frac{\eps \pi^2 n^2}{\sqrt{2}}.
    \end{equation}
    Then bounding each term in the definition of $H_{\eps}$ in~\eqref{eq:infinite_boundary_interior_kernel} by its absolute value, we obtain
    \begin{align}
        \left|
            H_{\eps}(t,y)
        \right|
        &
        \leq
        \frac{e^{1 / (\sqrt{2} \eps)}}{2}
        \sum_{n=1}^{\infty}
        e^{-\eps \pi^2 n^2 (t - 1 / \sqrt{2})}
        \\
        &
        \leq
        \frac{e^{1 / (\sqrt{2} \eps)}}{2}
        \int_{0}^{\infty}
        e^{-\eps \pi^2 (t - 1 / \sqrt{2}) x^2}
        \, dx
        \\
        &
        =
        \frac{e^{1 / (\sqrt{2} \eps)}}{4 \sqrt{\eps \pi (t - 1 / \sqrt{2})}}.
    \end{align}
    Plugging this into~\eqref{eq:infinite_hs_bound} yields
    \begin{equation}
        A_{\eps}^{-1}(T)
        \leq
        \frac{\eps}{2 \pi (T - 1 / \sqrt{2})}
        \exp
        \left(
            -\frac{T - 2 - 2 \sqrt{2}}{2 \eps}
        \right).
    \end{equation}
    Combining this estimate with~\eqref{eq:cost_bound_by_A} proves~\eqref{eq:negative_cost_bound} with
    \begin{equation}
        K_T
        =
        \sqrt{\frac{C_T}{2 \pi (T - 1 / \sqrt{2})}}.
    \end{equation}
\end{proof}

Combining Theorem~\ref{prop:negative_cost_bound} with the lower bound proved in Section~\ref{sec:negative_lower} completes the proof of Theorem~\ref{thm:main_negative}.

\section*{Acknowledgements}
Kai Koike has been supported by JSPS KAKENHI Grant Number 25K07077. The authors thank Franck Sueur for helpful discussion and valuable comments.

\section*{Statements and Declarations}

\noindent \textbf{Competing interests.}
The authors declare no competing interests.

\noindent \textbf{Data availability.}
No datasets were generated or analyzed during the current study.

\noindent \textbf{Use of generative AI.}
OpenAI's ChatGPT suggested the strategy for the positive-speed counterexample, including the use of model spaces. The upper-bound arguments were developed through discussions between the authors, with the AI assisting in carrying out calculations and refining the estimates. The negative-speed counterexample was constructed by the authors. The tool was also used for literature searches and manuscript preparation. The authors take full responsibility for the mathematical content and presentation of the final manuscript.

\bibliography{coron_guerrero_full}

\end{document}